\documentclass[11pt]{amsart}
\author{Pete L. Clark}
\author{Frederick Saia}

\title{CM elliptic curves: volcanoes, reality and applications, Part II}
\usepackage{amsmath,amssymb,amsthm,mathtools,url}
\usepackage{geometry,graphicx}
\usepackage{tikz, rotating}
\usepackage{float}
\usepackage{color}
\DeclareMathAlphabet{\curly}{U}{rsfs}{m}{n}
\newtheorem{thm}{Theorem}[section]

\newtheorem{example}[thm]{Example}
\newtheorem{cor}[thm]{Corollary}
\newtheorem{prop}[thm]{Proposition}
\newtheorem{lemma}[thm]{Lemma}

\theoremstyle{definition}

\newtheorem{remark}[thm]{Remark}
\begin{document}
\newcommand\leg{\genfrac(){.4pt}{}}
\renewcommand{\labelenumi}{(\roman{enumi})}
\def\ff{\mathfrak{f}}
\def\N{\mathbb{N}}
\def\Q{\mathbb{Q}}
\def\Z{\mathbb{Z}}
\def\R{\mathbb{R}}
\def\OO{\mathcal{O}}
\def\aa{\mathfrak{a}}
\def\pp{\mathfrak{p}}
\def\qq{\mathfrak{q}}
\def\Aa{\curly{A}}
\def\Dd{\curly{D}}
\def\Gg{\curly{G}}
\def\Pp{\curly{P}}
\def\Ss{\curly{S}}
\def\End{\mathrm{End}}
\def\M{\curly{M}}
\newcommand{\Ok}{\Z_K}
\newcommand{\tors}{\operatorname{tors}}
\newcommand{\Pic}{\operatorname{Pic}}
\newcommand{\ord}{\operatorname{ord}}
\newcommand{\GL}{\operatorname{GL}}
\newcommand{\PGL}{\operatorname{PGL}}
\newcommand{\Gal}{\operatorname{Gal}}
\newcommand{\gcdl}{\operatorname{gcd}}
\renewcommand{\gg}{\mathfrak{g}}

\newcommand{\abedit}[1]{{\color{blue} \sf  #1}}
\newcommand{\abbey}[1]{{\color{blue} \sf $\clubsuit\clubsuit\clubsuit$ Abbey: [#1]}}
\newcommand{\margAb}[1]{\normalsize{{\color{red}\footnote{{\color{blue}#1}}}{\marginpar[{\color{red}\hfill\tiny\thefootnote$\rightarrow$}]{{\color{red}$\leftarrow$\tiny\thefootnote}}}}}
\newcommand{\Abbey}[1]{\margAb{(Abbey) #1}}
\newcommand{\pcedit}[1]{{\color{brown} \sf #1}}
\newcommand{\ra}{\rightarrow}
\newcommand{\C}{\mathbb{C}}
\newcommand{\Ker}{\operatorname{Ker}}
\newcommand{\F}{\mathbb{F}}
\newcommand{\Spec}{\operatorname{Spec}}
\newcommand{\CM}{\operatorname{CM}}
\newcommand{\Aut}{\operatorname{Aut}}
\renewcommand{\aa}{\mathfrak{a}}
\newcommand{\bb}{\mathbf{b}}
\newcommand{\rad}{\operatorname{rad}}
\newcommand{\dd}{\mathfrak{d}}
\newcommand{\lcm}{\operatorname{lcm}}
\renewcommand{\P}{\mathbb{P}}
\newcommand{\cyc}{\operatorname{cyc}}
\newcommand{\PP}{\mathbb{P}}
\newcommand{\oddCM}{\operatorname{odd-CM}}
\newcommand{\cc}{\mathbf{c}}
\newcommand{\rr}{\mathfrak{r}}

\maketitle

\begin{abstract}
Let $M \mid N$ be positive integers, and let $\Delta$ be the discriminant of an order in an imaginary quadratic field $K$.  When $\Delta_K < -4$, the first author determined the fiber of the morphism $X_0(M,N) \ra X(1)$ over the closed point $J_{\Delta}$ 
corresponding to $\Delta$ and showed that all fibers of the map $X_1(M,N) \ra X_0(M,N)$ over $J_{\Delta}$ were connected. \cite{Clark22a}.  In the present work we complement the work of \cite{Clark22a} by addressing the most difficult cases $\Delta_K \in \{-3,-4\}$.  These works provide all the information needed to compute, for each positive integer $d$, all subgroups of $E(F)[\tors]$, where $F$ is a number field of degree $d$ and $E_{/F}$ is an elliptic curve with complex multiplication.
\end{abstract}

\tableofcontents

\section{Introduction}
\subsection{Main Results}
This paper is a direct continuation of \cite{Clark22a}, which determined the $\Delta$-CM locus on the modular 
curves $X_0(M,N)_{/\Q}$ for $\Delta$ the discriminant of an order in an imaginary quadratic field $K$ \emph{different from} $\Q(\sqrt{-1})$ or $\Q(\sqrt{-3})$.  This work also gave a completely explicit description of the  \emph{primitive} residue fields of $\Delta$-CM closed points: i.e., the residue fields $\Q(P)$ of $\Delta$-CM points $P \in X_0(M,N)$ for which there is no $\Delta$-CM point $P' \in X_0(M,N)$ such that $\Q(P')$ embeds into $\Q(P)$ as a proper subfield.  Finally, the work \cite{Clark22a} also gave an inertness result for the fibers of $X_1(M,N) \ra X_0(M,N)$ over $\Delta$-CM points on $X_0(M,N)$, which yields a complete description 
of the multiset of degrees of $\Delta$-CM closed points on $X_1(M,N)_{/\Q}$.  Using only the knowledge of the degrees of primitive residue fields of $\Delta$-CM points on $X_0(M,N)$ yields the corresponding knowledge of degrees of primitive residue fields of $\Delta$-CM 
points on $X_1(M,N)$, which is precisely what is needed in order to classify torsion subgroups of $\Delta$-CM elliptic curves 
over number fields of any fixed degree.   
\\ \\
In the present work we treat the excluded fields $\Q(\sqrt{-1})$ and $\Q(\sqrt{-3})$.  If $\Delta = \ff^2 \Delta_K$ 
and $\Delta_K \in \{-3,-4\}$, then for all $M \mid N$ we determine the $\Delta$-CM locus on $X_0(M,N)_{/\Q}$ and explicitly determine 
all primitive residue fields of $\Delta$-CM points on $X_0(M,N)_{/\Q}$.  Finally, we show that the fibers of $X_1(M,N) \ra X_0(M,N)$ over $\Delta$-CM points are connected.\footnote{The difference between ``inertness'' and ``connectedness'' is that the latter 
allows ramification.  The map $X_0(M,N) \ra X(1)$ can only ramify over $0$, $1728$ and $\infty$.}
\\ \\
Taken together, the works \cite{Clark22a} and the present work give a complete description of torsion subgroups of CM elliptic curves 
over number fields.  In particular, we get an algorithm that takes as input a positive integer $d$ and outputs the complete, 
finite list of groups isomorphic to $E(F)[\tors]$ where $F$ is a number field of degree $d$ and $E_{/F}$ is a CM elliptic curve, conditionally on knowing the finite list of imaginary quadratic orders of class number properly dividing $d$.  With current knowledge about class numbers, this allows us to enumerate CM torsion in number field degree $d \leq 200$.  If we are willing to assume the Generalized 
Riemann Hypothesis (GRH), then by \cite[Cor. 1.3]{LLS15} we can enumerate CM torsion in number field degree $d \leq 18104$.  
This enumeration will appear in a future work.

\subsection{Review of the $\Delta_K < -4$ case}
Let us outline the proof of the computation of the fiber of $X_0(M,N) \ra X(1)$ over the closed $\Delta$-CM point $J_{\Delta}$ on $X(1)$ given in \cite{Clark22a} so that we can see what must be modified to treat the $\Delta_K \in \{-3,-4\}$ case.
\\ \\
\textbf{Step 1}: We handle the case $X_0(1,\ell^a) = X_0(\ell^a)$ for a prime power $\ell^a$. \\
Step 1a: The $\ell$-power isogeny graph corresponding to a $\Delta$-CM elliptic curve with $\Delta = \ell^{2L} \ff_0^2 \Delta_K$ (where $\gcd(\ell,\ff_0) = 1$) has the structure of an 
\textbf{$\ell$-isogeny volcano}.  This is an infinite graph that is very close to being a rooted tree and with vertex set stratified into levels indexed by $\Z^{\geq 0}$; the set of vertices at level $L$ corresponds to the set of $j$-invariants of $(\Delta = \ell^{2L} \ff_0^2\Delta_K)$-CM elliptic curves, which is a torsor under $\Pic \OO(\Delta)$, the Picard group of the imaginary quadratic order $\OO(\Delta)$ of 
discriminant $\Delta$.   
Cyclic $\ell^a$-isogenies with $\Delta$-CM source elliptic curve correspond to paths of length $a$ in this volcano with initial 
vertex at level $L$.  From this it is easy to see that if $\varphi: E \ra E'$ is such an isogeny and if the target elliptic curve has level $L'$, 
then over $K$ the field of moduli of $\varphi$ is the ring class field $K(\ff_0 \ell^{\max(L,L')})$, which is equal to $K(j(E),j(E'))$.  It follows that 
\[ \Q(j(E),j(E')) \subseteq \Q(\varphi) \subseteq K(j(E),j(E')). \]
Step 1b: Thus the field of moduli $\Q(\varphi)$ of $\varphi$ is determined when $\Q(j(E),j(E'))$ contains $K$: this happens if and only if $\Q(j(E),j(E'))$ has no real embedding ($j(E)$ and $j(E')$ are ``not coreal'').  Otherwise we are left to decide whether $\Q(\varphi)$ is $\Q(j(E),j(E'))$ or $K(j(E),j(E'))$.  Both the coreality question and the dichotomy between the two possible fields can 
be answered in terms of the natural action of complex conjugation on the $\ell$-isogeny volcano.  Determining the explicit action of 
$\gg_{\R} = \{1,c\}$ on the $\ell$-isogeny volcano is one of the main contributions of \cite{Clark22a}.  In the end, depending upon whether or not the path is fixed under complex conjugation or not,\footnote{Since closed points on $X_0(\ell^a)$ correspond 
to certain equivalence classes of paths, the actual answer is slightly more complicated than this but can still be determined from the action of 
complex conjugation on the isogeny graph.} we get that $\Q(\varphi)$ is isomorphic to $\Q(\ff_0 \ell^{\max(L,L'})$ -- 
that is, isomorphic to a \emph{rational ring class field} -- the field obtained by adjoining to $\Q$ the $j$-invariant of an elliptic curve with CM by the imaginary quadratic 
order of discriminant $(\ff_0 \ell^{\max(L,L')})^2 \Delta_K$ -- or to the ring class field $K(\ff_0 \ell^{\max(L,L')})$ -- which is obtained 
by adjoining to $K$ the same $j$-invariant.
\\ \\
\textbf{Step 2:}  We pass from the prime power case $X_0(\ell^a)$ to the case $X_0(N)$. \\
Step 2a: For any closed point $p \notin \{0,1728,\infty\}$ on the $j$-line $X(1)$, if $N = \ell_1^{a_1} \cdots \ell_r^{a_r}$, let 
$F$ be the fiber of $\pi: X_0(N) \ra X(1)$ over $p$, and for $1 \leq i \leq r$ let $F_i$ be the fiber of $X_0(\ell_i^{a_i}) \ra X(1)$ 
over $p$.  Then we show that $F$ is the fiber product of $F_1,\ldots,F_r$ over $\Spec \Q(p)$.  Since each $F_i$ is the spectrum of a 
finite product of number fields, each isomorphic to either a rational ring class field or a ring class field, $F$ is determined by $F_i$ 
in terms of tensor products of these rational ring class fields and ring class fields. \\
Step 2b: Letting $\Q(\ff) \coloneqq \Q(j(\C/\OO(\ff^2 \Delta_K)))$ be the rational ring class field of conductor $\ff$, we show that 
\begin{equation}
\label{INTROEQ1}
 K(\ff_1) \otimes_{K(\gcd(\ff_1,\ff_2)} K(\ff_2) = K(\lcm(\ff_1,\ff_2)) 
\end{equation}
and
\begin{equation}
\label{INTROEQ2}
 \Q(\ff_1) \otimes_{\Q(\gcd(\ff_1,\ff_2)} \Q(\ff_2) = \Q(\lcm(\ff_1,\ff_2)). 
\end{equation}
These identities allow us to write down $F$ explicitly as a product of number fields.  
\\ \\ 
\textbf{Step 3:} Lifting a point $P$ on $X_0(N)$ induced by an isogeny $\varphi: E \ra E'$ to a point $\tilde{P}$ $X_0(M,N)$ involves 
scalarizing the modulo $M$ Galois representation on $E$, with the effect that $\Q(\tilde{P})$ is obtained from $\Q(P)$ by adjoining 
the projective $M$-torsion field $\Q(P)(\P E[M])$.  This uses that because $\Delta < -4$, the projective $M$-torsion field 
is independent of the choice of $\Q(P)$-rational model.  Indeed, for all $M \geq 3$, if $E$ is a $\Delta = \ff^2 \Delta_K$-CM elliptic curve, 
we find that $\Q(P)( \P E[M]) = \Q(P)K(M\ff)$.  

\subsection{The $\Delta_K \in \{-3,-4\}$ case} When $\Delta_K \in \{-3,-4\}$ and $E$ is a $\Delta = \ff^2 \Delta_K$-CM elliptic curve, then for certain $\ff > 1$ \emph{some} 
of the above steps still hold.  However, when $\ff = 1$, \emph{none} of the above arguments hold as stated.  While some of the change are routine, in several places 
we have to make arguments that are significantly more intricate than those of \cite{Clark22a}.  Let us describe the modifications:
\\ \\
Step 1a: When $\Delta_K \in \{-3,-4\}$, $\ell$ is a prime numer, and $\ff_0 > 1$, then again the $\ell$-power 
isogeny graph of a $(\Delta = (\ell^L \ff_0)^2 \Delta_K)$-CM elliptic curve is an $\ell$-volcano.  When $\ff_0 = 1$, the $\ell$-power 
isogeny graph is no longer an $\ell$-volcano.  This is in fact the least of our worries, as the deviation from ``volcanoness'' is minor and had already been well understood: the differences involve multiple edges descending from the surface and in subtleties involving orientations of edges which necessitate more care in the notion of a ``nonbacktracking path.'' The structural information we need is found, for instance, in \cite{Sutherland12}.
\\ \\
Step 1b: When $\ff_0^2 \Delta_K \in \{-3,-4\}$, there is in fact \emph{no} canonical action of complex conjugation on the 
$\ell$-isogeny graph.   To define the action of complex conjugation on isogenies, we need \emph{a priori} a chosen $\R$-model 
on the elliptic curve.  Every real elliptic curve has precisely two nonisomorphic $\R$-models.  When $\Delta < -4$ these two $\R$-models 
are quadratic twists of each other, so the action of $\gg_{\R}$ on finite subgroup schemes of $E_{/\R}$ is independent of the choice
of $\R$-model.  However, when $\Delta \in \{-3,-4\}$ this is no longer the case.  \\ \indent
It turns out that when $\ff_0^2 \Delta_K \in \{-3,-4\}$, in most cases we can define a \emph{noncanonical} action of $\gg_{\R}$ 
on the $\ell$-isogeny graph that allows us to determine fields of modulo of $\Delta$-CM points on $X_0(N)$ in terms of real and complex paths, as in \cite{Clark22a}.  But there are two cases in which we need to pass from this isogeny graph to a certain double cover and 
define an action of $\gg_{\R}$ on that.  By looking carefully on how surface edges change the $\R$-structure of a $\Delta_K$-CM elliptic 
curve we are able to carry out the analysis as in \cite{Clark22a}.
\\ \\
Step 2a: The fiber product result referred to above \cite[Prop. 3.5]{Clark22a} is false when $j \in \{0,1728\}$: the modular curve 
$X_0(\ell_1^{a_1} \cdots \ell_r^{a_r})$ is a desingularization of the fiber product of the morphisms $X_0(\ell_i^{a_i}) \ra X(1)$.  For lack of this scheme-theoretic result we need other techniques to compute the composite level fibers.  We make use of the Atkin-Lehner involution $w_N$ to reduce to either the $\Delta_K < -4$ case 
or the case when both the source and target are $\Delta_K$-CM, in which case we have a \emph{proper} isogeny in the sense of 
\cite[\S 3.4]{Clark22a} and we can reason via $\Z_K$-ideals.  
\\ \\
Step 2b: When $\Delta \in \{-3,-4\}$, equations (\ref{INTROEQ1}) and (\ref{INTROEQ2}) hold for certain pairs $\ff_1,\ff_2 \in \Z^+$ 
and fail for others: in general the compositum $K(\ff_1)K(\ff_2)$ is a proper subfield of $K(\lcm(\ff_1,\ff_2))$; and the same holds 
for rational ring class fields $\Q(\ff_1)$ and $\Q(\ff_2)$.  In \S 2 we use class field theory to show that it is still the case that $K(\ff_1)$ and $K(\ff_2)$ are linearly disjoint over $K(\gcd(\ff_1,\ff_2))$ and to determine the index of $K(\ff_1)K(\ff_2)$ in $K(\lcm(\ff_1,\ff_2))$; we do the same for rational ring class fields; and again we calculate tensor products of rational ring class fields and ring class fields.  This calculation is relatively straightforward, but
the difference in the answer causes complications related to Step 2a: when $\Delta_K < -4$, if $N_1,N_2$ are coprime positive integers, 
to show that the residue field $\Q(P)$ of a point $P \in X_0(N)$ contains  e.g. a ring class field $K(N_1 N_2)$, it suffices to show 
that it contains each of $K(N_1)$ and $K(N_2)$.   When $\Delta_K \in \{-3,-4\}$, we cannot argue in this way.  Instead our method is to 
find a rationally isogenous elliptic curve with CM conductor divisible by $N_1N_2$.  
\\ \\
Step 3: When $\Delta \in \{-3,-4\}$, the projective $M$-torsion field $F( \P E[M])$ of a $\Delta$-CM elliptic curve may depend upon the model.  Nevertheless we compute the residue field $\Q(P)$ for any $\Delta$-CM point $P \in X_0(M,M)$ (this amounts to computing 
the minimal possible projective $M$-torsion field as we range over models). Using this and the maps $\alpha: X_0(M,N) \ra X_0(M,M)$ and $\beta: X_0(M,N) \ra X_0(N)$, we can bootstrap from $\beta(P) \in X_0(N)$ to $P \in X_0(M,N)$, with some care: the case 
$\Delta = -4$ behaves exceptionally to all the rest.

\subsection{The CM fibers of $X_1(M,N) \ra X_0(M,N)$ are connected}\label{X_1_sec}

In \cite[Thm. 1.2]{Clark22a}, the first author showed an especially close relationship between points on $X_0(M,N)$ and points on $X_1(M,N)$ for points which do not have CM by $\Delta \in \{-3,-4\}$. In particular, this theorem states that the fiber of the map $X_1(M,N) \rightarrow X_0(M,N)$ is inert over any point which does not have CM by one of these two discriminants.  This has the important consequence that determining the degrees of closed $\Delta$-CM points on $X_0(M,N)$ and on $X_1(M,N)$ are \emph{equivalent problems}.  The following theorem generalizes this result to include points with $-3$ and $-4$-CM.

\begin{thm}\label{X0_to_X1}
Let $M \mid N \in \mathbb{Z}^+$, and suppose that $x \in X_0(M,N)_{/\mathbb{Q}}$ is a $\Delta$-CM point. Let $\pi: X_1(M,N) \rightarrow X_0(M,N)$ denote the natural morphism. 
\begin{enumerate}
\item If $\Delta < -4$ or if $M \geq 2$, then $\pi$ is inert over $x$. 
\item Suppose that $\Delta \in \{-3,-4\}$ and $M = 1$. 
\begin{enumerate}
\item If $x$ is a ramified point of the map $X_0(M,N) \rightarrow X(1)$ or if $N \leq 3$, then $\pi$ is inert over $x$. 
\item Otherwise, i.e., if $N \geq 4$ and $x$ is an elliptic point on $X_0(M,N)$, then we have 
\[ e_{\pi}(x) = \begin{cases} 2 \quad &\text{ if } \Delta = -4 \\ 3 \quad &\text{ if } \Delta = -3 \end{cases} \quad \text{ and } \quad f_{\pi}(x) = \begin{cases} \phi(N)/4 \quad &\text{ if } \Delta = -4 \\ \phi(N)/6 \quad &\text{ if } \Delta = -3 \end{cases} \]
for the ramification index and residual degree of $x$. 
\end{enumerate}
\end{enumerate}
In particular, in all cases we have that the fiber of $\pi$ over $x$ consists of a single point. 
\end{thm}
\begin{proof}
For the proof, we first recall some basic relevant facts: for $N \leq 2$ the map $\pi$ is an isomorphism. For $N \geq 3$ it is a $(\mathbb{Z}/N\mathbb{Z})^*/\{\pm 1\}$-Galois covering, hence has degree $\phi(N)/2$. All points on $X(N) = X_1(N,N)$ not above $0, 1728 \in X(1)$ are unramified. For $N \geq 4$ the curve $X_1(N)$ (and hence $X(N)$) has no elliptic points of periods $2$ or $3$, from which it follows that all points over $0, 1728 \in X(1)$ are ramified with ramification index $2$ or $3$. The curve $X_1(2)$ has a single elliptic point of period $2$ over $1728 \in X(1)$, while the curve $X_1(3)$ has a single elliptic point of period $3$ over $0 \in X(1)$. (One can see these claims regarding elliptic points and ramification from elementary arguments involving congruence subgroups, in fact this is \cite[Exc. 2.3.7]{DS05}). 

For $\Delta < -4$ the claim is \cite[Thm 1.2]{Clark22a}, so suppose that $\Delta \in \{-3,-4\}$. For $M \geq 2$, the point $x$ must be non-elliptic (i.e., is a ramified point of the map $X_0(M,N) \rightarrow X(1)$). We can see this, for instance, via our analysis of paths on $\mathcal{G}_{K,\ell,1}$ for any prime $\ell \mid M$; in all cases we find that any pair of independent $\ell^{a'}$ isogenies for $a' \geq 1$ must include at least one with a corresponding path in $\mathcal{G}_{K,\ell,1}$ which descends, and hence any $-3$ or $-4$-CM point on $X_0(\ell^{a'},\ell^{a})$, and hence on $X_0(M,N)$ for $M \geq 2$, must be non-elliptic. In this case we then have that a pair $(E,C)_{/\mathbb{Q}(x)}$ inducing $x$ is well-defined up to quadratic twist, as all models for $E$ are defined over $\mathbb{Q}(x)$. For this reason, the same argument involving the modulo $N$ $\pm$-Galois representation given in \cite[Thm 1.2]{Clark22a} applies. Similarly, this argument applies in case $(2)(a)$ if $x$ is a ramified point of the map $X_0(M,N) \rightarrow X(1)$. 

We now assume that $x$ is an elliptic $\Delta$-CM point on $X_0(M,N)$ with $\Delta \in \{-3,-4\}$. If $N = 2$, then $\pi$ is an isomorphism, so the claim is trivial. If $N=3$, then because there is a single elliptic point on $X_1(3)$ it follows that it must comprise the entire fiber above $x$, giving the inertness claim. Assuming now that $N \geq 4$, we know that $x$ is elliptic while every point in $\pi^{-1}(x)$ is ramified with respect to the map $X_1(N) \rightarrow X(1)$, giving the claimed ramification index. Note that it follows that the residual degree is at most the claimed residual degree in each case. 

To provide the lower bound on the residual degree, we need only modify the argument of the $\Delta < -4$ case slightly in a predictable way. If $\Delta = -4$, then a pair $(E,C)_{\mathbb{Q}(x)}$ inducing $x$ is well-defined up to quartic twist. Letting $q_N : \Z_K \rightarrow \Z_K/N\Z_K$ denote the quotient map, by tracking that action of Galois on a generator $P$ of $C$ we get a well-defined reduced mod $N$ Galois representation 
\[ \overline{\rho_N} : \gg_{\Q(x)} \rightarrow \left(\Z_K / N\Z_K\right)^\times / q_N(\Z_K^*) \]
which is independent of the chosen model and surjective (see \cite[\S 1.3]{BCI}). As the set $\{P,-P,iP,-iP\}$ is stable under the action of $\gg_{\Q(y)}$ for $y \in \pi^{-1}(x)$, we then must have 
\[ \frac{\phi(N)}{4} = \# \left( \overline{\rho_N}\left(\gg_{\Q(x)}\right) \right) \mid [\mathbb{Q}(y) : \mathbb{Q}(x)], \]
giving the result for $\Delta = -4$. For $\Delta = -3$, exchanging ``quartic'' for ``cubic'' and $\mu_4$ for $\mu_3$ results in the required divisibility $\frac{\phi(N)}{6} \mid f_\pi(x)$. 

\end{proof}

\begin{remark}
The $M=1$ with $\Delta < -4$ case of Theorem \ref{X0_to_X1} is used explicitly in \cite{CGPS22} to transfer from knowledge of the \emph{least} degree of a $\Delta$-CM point on $X_1(N)$, which is computed in \cite{BCII}, to knowledge of the least degree of a $\Delta$-CM point on $X_0(N)$. A shadow of Theorem \ref{X0_to_X1} is also seen in the referenced study in the $\Delta \in \{-3,-4\}$ case. A positive integer $N$ is of Type I or Type II, using the terminology of \cite{CGPS22}, if $X_0(N)$ has an elliptic point of order $3$ or $2$, respectively. If $N$ is of type I, then there is a single primitive degree among all elliptic points on $X_0(N)$ which is the least degree of a $-4$-CM point on $X_0(N)$. In this case, the single point lying above any elliptic $-4$-CM point on $X_1(N)$ provides the least degree of a $-4$-CM point on $X_1(N)$ (and the analogous statements hold for Type II and $\Delta = -3$). 
\end{remark}


\section{Composita of Ring Class Fields and of Rational Ring Class Fields}
\noindent
Let $K$ be an imaginary quadratic field, of discriminant $\Delta_K$.  We put 
\[ w_K \coloneqq \# \Z_K^{\times} = \begin{cases} 6 & \text{ if }  \Delta_K = -3 \\ 4 &  \text{ if }  \Delta_K = -4 \\ 2 &  \text{ if }  \Delta_K < -4 \end{cases}. \]
Let $\OO$ be a $\Z$-order in $K$.  For $\ff \in \Z^+$, there is a unique $\Z$-order $\OO$ in $K$ with $[\Z_K:\OO] = \ff$ and then $\ff \Z_K$ is the conductor ideal $(\OO:\Z_K)$ \cite[\S 2.1]{Clark22a}.
We denote by $K(\ff)$ the \textbf{ring class field} of $\OO$ \cite[\S 2.3]{Clark22a}.  If $j_{\Delta} \coloneqq j(\C/\OO)$, then we have 
\[ K(\ff) = K(j_{\Delta}). \]
We recall from \cite[Cor. 7.24]{Cox89} the formula

\begin{equation}
\label{QUADCLASS}
\mathfrak{d}(\ff) \coloneqq [K(\ff):K(1)] = \begin{cases} 1 &  \text{ if }  \ff = 1 \\
\frac{2}{\# \Z_K^{\times}} \ff \prod_{\ell \mid \ff} \left(1 - \left( \frac{\Delta_K}{\ell} \right) \frac{1}{\ell} \right) &  \text{ if }  \ff \geq 2
\end{cases}.
\end{equation}
As in \cite[\S 2.6]{Clark22a}, we also define the \textbf{rational ring class field}
\[ \Q(\ff) \coloneqq \Q(j_{\Delta}). \]
In \cite[\S 2]{Clark22a} we studied composita of ring class fields and of rational ring class fields (with a fixed imaginary quadratic field 
$K$, in both cases) when $\Delta_K< -4$.  The results were quite clean: for $\ff_1,\ff_2 \in \Z^+$, the fields 
$K(\ff_1)$ and $K(\ff_2)$ are linearly disjoint over $K(\gcd(\ff_1,\ff_2))$ and we have $K(\ff_1)K(\ff_2) = K(\lcm(\ff_1,\ff_2))$ 
\cite[Prop. 2.2]{Clark22a} and the same holds with each $K(\ff_i)$ replaced by $\Q(\ff_i)$ \cite[Prop. 2.10a)]{Clark22a}.
\\ \\
Here we treat $\Delta_K \in \{-3,-4\}$.

\begin{prop}
\label{TEDIOUSALGEBRAPROP1}
\label{SAIAPROP1}
Let $K$ be a quadratic field with $\Delta_K \in \{-3,-4\}$, let $\ff_1,\ff_2 \in \Z^+$, and put 
\[ m \coloneqq \operatorname{gcd}(\ff_1,\ff_2), \ M \ \coloneqq \operatorname{lcm}(\ff_1,\ff_2). \]
\begin{itemize}
\item[a)] Suppose that $m > 1$.  Then:
\[ K(\ff_1)K(\ff_2) = K(M).\]
\item[b)] If the order of discriminant $\ff_1^2 \Delta_K$ has class number $1$, then we have 
\[ K(\ff_1)K(\ff_2) = K(\ff_2). \]
\item[c)] Let $\ff_1,\ldots,\ff_r \in \Z^+$ be pairwise relatively prime, and further assume that: \\
$\bullet$ If $\Delta_K = -3$, then no $\ff_i$ lies in $\{1,2,3\}$; and \\
$\bullet$ If $\Delta_K = -4$, then no $\ff_i$ lies in $\{1,2\}$.  \\
Then: 
\[ [K(\ff_1 \cdots \ff_r): K(\ff_1) \cdots K(\ff_r)] = \left(\frac{w_K}{2}\right)^{r-1}. \]
%
\item[d)] In all cases we have that $K(\ff_1)$ and $K(\ff_2)$ are linearly disjoint over $K(m)$, and thus $K(\ff_1) \cap K(\ff_2) = K(m)$.
\end{itemize}
\end{prop}
\begin{proof}
We will use the classical description of ring class groups and ring class fields, with notation as in \cite[\S 7]{Cox89}.  For $\ff \in \Z^+$, 
let $I_K(\ff)$ be the group of fractional $\Z_K$-ideals prime to $\ff$ and let $P_{K,\Z}(\ff)$ be the subgroup of principal fractional ideals 
generated by an element $\alpha \in \Z_K$ such that $\alpha \equiv a \pmod{\ff \Z_K}$ for some $a \in \Z$ with $\gcd(a,\ff) = 1$.  
By class field theory, we have $K(\ff_1)K(\ff_2) = K(M)$ if and only if 
\[ P_{K,\Z}(\ff_1) \cap P_{K,\Z}(\ff_2) = P_{K,\Z}(M). \]
Clearly in all cases we have 
\[ P_{K,\Z}(\ff_1) \cap P_{K,\Z}(\ff_2) \supseteq P_{K,\Z}(M). \]
a) 
$\bullet$ Suppose $\Delta_K = -4$ and $m  > 1$, so the units of $\Z_K$ are $\pm 1, \pm \sqrt{-1}$.   Let $(\alpha) \in P_{K,\Z}(\ff_1) \cap 
P_{K,\Z}(\ff_2)$.  We may choose $\alpha$ such that 
\[ \alpha \equiv a_{\ff_1} \pmod{ \ff_1 \Z_K} \]
and then there is $u \in \Z_K^{\times}$ such that 
\[ u \alpha \equiv a_{\ff_2} \pmod{\ff_2 \Z_K}. \]
If $u \in \{ \pm 1\}$, then the argument of Case 1 works to show that $(\alpha) \in P_{K,\Z}(M)$.  After replacing $\alpha$ with $-\alpha$ 
if necessary, the other case to consider is that 
\[ \sqrt{-1} \alpha \equiv a_{\ff_2} \pmod{\ff_2 \Z_K}. \]
If this holds then
\[ \frac{a_{\ff_2}}{a_{\ff_1}} \equiv i \pmod{m \Z_K}, \]
which is manifestly false. \\ 
$\bullet$ Suppose $\Delta_K = -3$ and $m > 1$, so the units of $\Z_K$ are $\pm 1, \pm \omega$, $\pm \overline{\omega}$, 
where $\omega = \frac{1+\sqrt{-3}}{2}$.  As above, we may suppose that $\alpha \equiv a_{\ff_1} \pmod{\ff_1 \Z_K}$ and 
$\alpha$ is congruent modulo $\ff_2 \Z_K$ to either $\omega a_{\ff_2}$ or to $\overline{\omega} a_{\ff_2}$.  We then get 
\[ \frac{a_{\ff_2}}{a_{\ff_1}} \equiv \omega \text{ or } \overline{\omega} \pmod{m \Z_K}, \]
which is again manifestly false.
\\
b) This is a trivial case, listed for completeness: if the order of discriminant $\ff_1^2 \Delta_K$ has class number $1$ 
then $K(\ff_1) = K(1)$ (and conversely), so $K(\ff_1) K(\ff_2) = K(1) K(\ff_2) = K(\ff_2)$.\footnote{This gives rise to cases in  
in which $K(M) \supsetneq K(\ff_1)K(\ff_2)$: e.g. when $\Delta_K = -3$ we have $K(2)K(3) = K(1)$ but $[K(6):K(1)] = 3$.}
\\
c) We claim that the extensions $K(\ff_1),\ldots,K(\ff_r)$ are mutually linearly disjoint over $K(1)$: that is, \[K(\ff_1) \otimes_{K(1)} \cdots 
\otimes_{K(1)} K(\ff_r) = K(\ff_1) \cdots K(\ff_r). \]
Since everything in sight is Galois, it is enough to check that $(K(\ff_1) \cdots K(\ff_{r-1})) \cap K(\ff_r) = K(1)$.  
But the conductor of $K(\ff_1) \cdots K(\ff_{r-1})$ divides $\ff_1 \cdots \ff_{r-1}$ and the conductor of $K(\ff_r)$ divides $\ff_r$, 
so the conductor of their intersection is the unit ideal, so the intersection is contained in the Hilbert class field $K(1)$, hence is 
equal to $K(1)$.  From this it follows that 
\[ [K(\ff_1 \cdots \ff_r):K(\ff_1) \cdots K(\ff_r)] = \frac{\delta(\ff_1 \cdots \ff_r)}{\prod_{i=1}^r \delta(\ff_i)}, \]
and the latter expression may be evaluated using (\ref{QUADCLASS}).
\\
d) It is immediate that $K(m) \subseteq K(\ff_1) \cap K(\ff_2)$.  \\ \indent
The case $m = 1$ is easy: then $K(\ff_1) \cap K(\ff_2)$ has conductor dividing $\ff_1$ and $\ff_2$, so its conductor is the unit 
ideal, so $K(\ff_1) \cap K(\ff_2)$ is contained in the Hilbert class field of $K$, which is the ring class field $K(1)$.  \\ \indent
Henceforth we suppose that $m > 1$, and thus by part a) we have $K(\ff_1)K(\ff_2) = K(M)$.  We claim the formula 
\[ \mathfrak{d}(m) \mathfrak{d}(M) = \mathfrak{d}(\ff_1) \mathfrak{d}(\ff_2). \]
First we observe that this formula $g(m)g(M) = g(\ff_1) g(\ff_2)$ holds for any multiplicative function $g: \Z^+ \ra \C$.  
If we had $\Delta_K < -4$ then the function $\mathfrak{d}$ would be multiplicative.  Instead we have $\Delta_K \in \{-3,-4\}$, in which case $\mathfrak{d}$ is a constant multiple of a multiplicative function \emph{except for its value at $1$}.  This justifies the claim.  The claim can be rewritten as 
\[ [K(\ff_1)K(\ff_2):K(m)] = [K(M):K(m)] = [K(\ff_1):K(m)][K(\ff_2):K(m)], \]
so $K(\ff_1)$ and $K(\ff_2)$ are linearly disjoint over $K(m)$, and thus $K(m) = K(\ff_1) \cap K(\ff_2)$.  
\end{proof}

\begin{prop}
\label{SAIAPROP2}
Let $K$ be a quadratic field with $\Delta_K \in \{-3,-4\}$.  Let $\ff_1,\ff_2 \in \Z^+$, and put $m = \gcd(\ff_1,\ff_2)$, $M = \lcm(\ff_1\ff_2)$.  Let
\[ \mathcal{D} \coloneqq \{-3,-4,-12,-16,-27\}; \]
this is the set of imaginary quadratic disciminants $\Delta = \ff^2 \Delta_K$ with fundamental discriminant $\Delta_K \in \{-3,-4\}$ 
and class number $1$.  Let
\[ S \coloneqq \{ \ff \in \Z^+ \mid \ff^2 \Delta_K \in \mathcal{D} \}. \]  
\begin{itemize}
\item[a)] The fields $\Q(\ff_1)$ and $\Q(\ff_2)$ are linearly disjoint over $\Q(m)$: \[\Q(\ff_1) \otimes_{\Q(m)} \Q(\ff_2) \cong \Q(\ff_1)\Q(\ff_2). \]
\item[b)] If $\ff_1 \in S$, then we have: 
\[\Q(\ff_1) \otimes_{\Q(m)} \Q(\ff_2) \cong \Q(\ff_2), \]
\[ \Q(\ff_1) \otimes_{\Q(m)} K(\ff_2) \cong K(\ff_1) \otimes_{\Q(m)} \Q(\ff_2) \cong K(\ff_2), \]
and 
\[ K(\ff_1) \otimes_{\Q(m)} K(\ff_2) \cong K(\ff_2) \times K(\ff_2). \]
\item[c)] If $\ff_1,\ff_2 \notin S$ and $m > 1$, then we have 
\[ \Q(\ff_1) \otimes_{\Q(m)} \Q(\ff_2) \cong \Q(\ff_1) \Q(\ff_2) =  \Q(M), \]
 \[ \Q(\ff_1) \otimes_{\Q(m)} K(\ff_2) \cong \Q(\ff_2) K(\ff_2) = K(M), \]
and 
\[ K(\ff_1) \otimes_{\Q(m)} K(\ff_2) \cong K(M) \times K(M). \]
\item[d)] Let $\ff_1,\ldots,\ff_r$ be elements of $\Z^+ \setminus S$ that are pairwise relatively prime.  
Then $\Q(\ff_1) \cdots \Q(\ff_r)$ is a subfield of $\Q(\ff_1 \cdots \ff_r)$ of index $\left(\frac{w_K}{2}\right)^{r-1}$, and moreover 
\[ \Q(\ff_1) \cdots \Q(\ff_r) \cong \Q(\ff_1) \otimes_{\Q(m)} \cdots \otimes_{\Q(m)} \Q(\ff_r), \]
\[ \Q(\ff_1) \otimes_{\Q(m)} \cdots \Q(\ff_{r-1}) \otimes_{\Q(m)} K(\ff_r) \cong \Q(\ff_1) \cdots \Q(\ff_{r-1}) K(\ff_r).\]
Finally, if $2 \leq s \leq r$, then 
\[ \Q(\ff_1) \otimes_{\Q(m)} \dots \otimes_{\Q(m)} \Q(\ff_{r-s}) \otimes_{\Q(m)} K(\ff_{r-s+1}) \otimes_{\Q(\ff)} \cdots 
\otimes_{\Q(\ff)} K(\ff_r) \cong \left(K(\ff_1) \cdots K(\ff_r)\right)^{2^{s-1}}. \]
\end{itemize}
\end{prop}
\begin{proof}
a) As in the proof of \cite[Prop. 2.10]{Clark22a}, this follows from the fact that $K(\ff_1)$ and $K(\ff_2)$ are 
linearly disjoint over $K(m)$.  \\
b) If $\ff_1 \in S$, then $\Q(m) = \Q(\ff_1) = \Q(1)$, and all the statements follow easily. \\
c) Using part a) and Proposition \ref{SAIAPROP1}a), we get
\[ [\Q(M):\Q(m)] = [K(M):K(m)] = [K(\ff_1)K(\ff_2):K(m)] = [K(\ff_1) \otimes_{K(m)} K(\ff_2):K(m)]  \]
\[ = [\Q(\ff_1) \otimes_{\Q(m)} \Q(\ff_2): \Q(m)] = [ \Q(\ff_1)\Q(\ff_2):\Q(m)], \]
so $\Q(\ff_1)\Q(\ff_2) = \Q(M)$.  The other two statements of part c) follow easily.  \\
d) Again it follows from Proposition \ref{SAIAPROP1} that the field extensions $\Q(\ff_1),\ldots,\Q(\ff_r)$ are mutually linearly disjoint 
over $\Q(1)$.  So
\[ [\Q(\ff_1) \cdots \Q(\ff_r):\Q(1)] = \prod_{i=1}^r [\Q(\ff_i):\Q(1)] = \left( \frac{w_K}{2} \right)^{1-r} [\Q(\ff_1 \cdots \ff_r):\Q(1)]. \]
The other two statements of part d) follow easily.
\end{proof}

\section{The Isogeny Graph $\mathcal{G}_{K,\ell,\ff_0}$}

\subsection{Defining the graph}
Let $K$ be an imaginary quadratic field, and let $\ell$ be a prime number.  There is a directed multigraph $\mathcal{G}_{K,\ell}$ as follows: 
the vertex set $\mathcal{V}$ of $\mathcal{G}_{K,\ell}$ is the set of $j$-invariants $j \in \C$ of \textbf{K-CM} elliptic curves, i.e., $j$-invariants of complex elliptic curves with endomorphism ring an order in the imaginary quadratic field $K$.  In general, for 
$j \in \mathcal{V}$ we denote by $E_j$ a complex elliptic curve with $j$-invariant $j$.  As for the edges:
let $\pi_1: X_0(\ell) \ra X(1)$ be the natural map, let $w_N \in \Aut( X_0(N))$ be the Atkin-Lehner involution, and let 
$\pi_2 \coloneqq \pi_1 \circ w_N$: here we work over $\C$.  For $j,j' \in \mathcal{V}$, write 
\[ (\pi_2)_* \pi_1^*([j]) = \sum_P e_P [P]. \] 
Then the number of directed edges from $j$ to $j'$ is $e_{j'}$.  
Equivalently, let $E_{/\C}$ be any elliptic curve with $j$-invariant $j$.  Then the number of edges from $j$ to $j'$ is the number 
of cyclic order $\ell$ subgroups $C$ of $E$ such that $j(E/C) = j'$.  
\\ \\
In \cite[\S 4]{Clark22a} we recalled the complete structure of the graph $\mathcal{G}_{K,\ell,\ff_0}$ when $\ff_0^2 \Delta_K < -4$ and saw in particular that it was an $\ell$-volcano in the sense of \cite[\S 4.2]{Clark22a}.  Now we need to describe the structure of $\mathcal{G}_{K,\ell,\ff_0}$ when 
$\ff_0^2 \Delta_K \geq -4$: i.e., when $\ff_0 = 1$ and $\Delta_K \in \{-3,-4\}$.

\subsection{$\Delta_K = -4$} Suppose $\Delta_K = -4$, $\ff_0 = 1$, and let $\ell$ be a prime number.

\begin{example}
\label{EXAMPLE4.2}
Let $K = \Q(\sqrt{-1})$, $\ell = 2$ and $\ff_0 = 1$.  The surface of this graph consists of CM $j$-invariants of discriminant $-4$, of 
which there is $1$: $j = 1728$.  Level one consists of CM $j$-invariants of discriminant $-16$, of which there is again $1$: 
$j = 2^3 \cdot 3^3 \cdot 11^3$.  Level two consists of CM $j$-invariants of discriminant $-64$, of which there are $2$.  As always, 
they form a single Galois orbit.  We have 
\[ J_{-64}(t) = t^2 - 82226316240t - 7367066619912. \]
There is one horizontal edge at the surface (a loop), corresponding to the unique $\Z[\sqrt{-1}]$-ideal $\pp_2$ of norm $2$.  The remaining 
two edges emanating outward from $j = 1728$ connect it to $j = 2^3 \cdot 3^3 \cdot 11^3$.  This corresponds to the fact that 
the pullback of the degree $1$ divisor $J_{1728}$ under $\pi: X_0(2) \ra X(1)$ is $[J_{1728}] + 2[J_{2^3 \cdot 3^3 \cdot 11^3}]$.  
\\ \indent
One of the three order $2$ subgroups of $E_{1728}$ is $E[\pp_2]$.  The other two are interchanged by the action of $\mu_4/\mu_2$ 
on $E_{1728}[2]$.  \\ 
The vertex $j = 1728$ has outward degree $3$ and inward degree $2$, while the vertex $j = 2^3 \cdot 3^3 \cdot 11^3$ has outward degree $3$ and inward degree $4$.  \\

\begin{figure}[H] 
\includegraphics[scale=0.5]{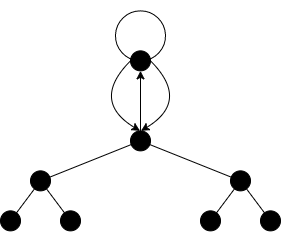}
\setlength{\abovecaptionskip}{10pt plus 0pt minus 0pt}
\caption{the graph $\mathcal{G}_{\mathbb{Q}(\sqrt{-1}), 2, 1}$ up to level $3$}\label{-4_ram_black}
\end{figure}

\end{example}

\noindent
Next suppose $\ell \equiv 1 \pmod{4}$.  Then there are two loops emanating from the surface vertex $v_0$, corresponding to the 
two prime ideals $\pp_1$, $\pp_2 = \overline{\pp_1}$ of $\Z[\sqrt{-1}]$ lying over $\ell$.   Let $v_1$ be any one of the $\frac{\ell-1}{2}$ 
level one vertices.  There are two directed edges from $v_0$ to $v_1$.  The natural action of $\mu_4/\mu_2$ on edges with emanating from $v_0$ fixes each of the two surface loops and interchanges the pair of edges from $v_0$ to $v_1$.  For each vertex at level $L \geq 1$ there is one upward 
edge and $\ell$ downward edges.
\\ \indent
Finally suppose $\ell \equiv 3 \pmod{4}$.  There are no surface edges.  For each vertex $v_1$ at level $1$ there are two edges 
from $v_0$ to $v_1$.  These two edges are interchanged by the $\mu_4/\mu_2$-action.  For each vertex at level $L \geq 1$ there 
is one upward edge and $\ell$ downward edges.

\subsection{$\Delta_K = -3$} Suppose $\Delta_K = -3$, $\ff_0 = 1$, and let $\ell$ be a prime number.

\begin{example}
\label{EXAMPLE4.3}
Let $K = \Q(\sqrt{-3})$, $\ell = 3$ and $\ff_0 = 1$.  The surface of this graph consists of CM $j$-invariants of discriminant $-3$, 
of which there is $1$: $j = 0$.  
Level one consists of CM $j$-invariants of discriminant $-3 \cdot 3^2$, of which there is again $1$: $j = - 2^{15} \cdot 3 \cdot 5^3$.   Level 
two consists of CM $j$-invariants of discriminant $-3 \cdot 3^4$, of which there are $3$, forming a single Galois orbit.  We have $J_{-3 \cdot 3^4} =$ \[ t^3 + 1855762905734664192000 t^2 -
    3750657365033091072000000 t + 3338586724673519616000000000. \]
There is one horizontal edge at the surface (a loop), corresponding to the unique $\Z[\frac{1+\sqrt{-3}}{2}]$-ideal $\pp_3$ of norm $3$.  
The remaining three edges emanating outward from $j = 0$ connect it to $j = -2^{15} \cdot 3 \cdot 5^3$.  This corresponds to the 
fact that the pullback of the degree $1$ divisor $J_0$ under $\pi: X_0(3) \ra X(1)$ is $[J_0] + 3 [J_{-2^{15} \cdot 3 \cdot 5^3}]$.  \\ \indent
One of the four order $3$ subgroups of $E_0$ is $E[\pp_3]$.   The other three are interchanged by the action of $\mu_6/\mu_2$ 
on $E_0[2]$.  \\ 
The vertex $j = 0$ has outward degree $4$ and inward degree $2$, while the vertex $j = - 2^{15} \cdot 3 \cdot 5^3$ has outward 
degree $4$ and inward degree $6$.  
\end{example}
\noindent
Next suppose $\ell \equiv 1 \pmod{3}$.  Then there are two loops emanating from the surface vertex $v_0$ corresponding to the two 
prime ideals of $\Z[\zeta_6]$ lying over $\ell$.  Let $v_1$ be any one of the $\frac{\ell-1}{3}$ level one vertices.  There are three directed edges from $v_0$ to $v_1$.  The natural action of $\mu_6/\mu_4$ on surface edges fixes each of the two surface loops and cyclically 
permutes the three edges from $v_0$ to $v_1$.  
\\ \\
Finally suppose $\ell \equiv 2 \pmod{3}$.  There are no surface edges.   For each vertex $v_1$ at level $1$ there are three edges 
from $v_0$ to $v_1$.  These edges are cyclically permuted by the $\mu_4/\mu_2$-action.  For each vertex at level $L \geq 1$ there 
is one upward edge and $\ell$ downward edges.

\subsection{Paths and $\ell^a$-isogenies}\label{paths_section} When $\ff_0^2 \Delta_K < -4$, \cite[Lemma 4.2]{Clark22a} gives a bijective correspondence 
between isomorphism classes of cyclic $\ell^a$ isogenies $\varphi: E \ra E'$ where $E_{/\C}$ is a $K$-CM elliptic curve for which the 
prime to $\ell$ part of the conductor of the endomorphism ring is $\ff_0$ and length $a$ nonbacktracking paths in $\mathcal{G}_{K,\ell,\ff_0}$.  In these cases every edge in $\mathcal{G}_{K,\ell,\ff_0}$ has a canonical inverse edge, so the directedness of $\mathcal{G}_{K,\ell,\ff_0}$ does not really intervene.  
\\ \indent
When $\ff_0^2 \Delta_K \in \{-3,-4\}$, the notion of a nonbacktracking path in $\mathcal{G}_{K,\ell,\ff_0}$ is a bit more subtle when the path involves ascent to and descent from the surface.  If we descend from any surface vertex $v_0$ to a level one vertex $v_1$ and then ascend back to $v_0$, then the latter edge must represent the dual isogeny of the former edge, since it is the \emph{unique} isogeny 
between these two elliptic curves, so this counts as backtracking.  On the other hand, if we start at a level one vertex $v_1$ take the unique edge $e: v_1 \ra v_0$ and then descend back down to $v_1$, we have a choice of $2$ edges when $\Delta_K = -4$ 
and $3$ edges when $\Delta_K = -3$.  Then $e$ corresponds to an $\ell$-isogeny $\varphi: E_1 \ra E_0$ and exactly one of the 
edges from $v_0$ to $v_1$, say $e'$, corresponds to $\varphi^{\vee}$.  So a path containing $e$ followed by $e'$ counts as backtracking, 
but a path containing $e$ followed by any other edge from $v_0$ to $v_1$ does not. \\ \indent With this understanding, \cite[Lemma 4.2]{Clark22a} 
extends to all $\Delta_K$, $\ell$ and $\ff_0$.

\begin{lemma}
\label{NEW4.2}
Let $K$ be an imaginary quadratic field, $\ell$ a prime number and $\ff_0$ a positive integer prime to $\ell$.  There is a bijective correspondence from the set of isomorphism classes of cyclic $\ell^a$-isogenies of CM elliptic curves with endomorphism algebra $K$ and prime-to-$\ell$-conductor $\ff_0$ to the set of length $a$ paths without backtracking in the isogeny graph $\mathcal{G}_{K,\ell,\ff_0}$.  
\end{lemma}
\noindent
Moreover the proof of \cite[Lemma 4.2]{Clark22a} still works to establish Lemma \ref{NEW4.2}.

\section{Action of Complex Conjugation on $\mathcal{G}_{K,\ell,\ff_0}$}
\noindent
This section is the analogue of \cite[\S 5]{Clark22a} for $\ff_0^2 \Delta_K \in \{-3,\-4\}$.  We define an action of $\gg_{\R} = \{1,c\}$ 
on the isogeny graph $\mathcal{G}_{K,\ell,\ff_0}$ that is crucial for our subsequent analysis...almost.  We will see that in two cases there 
is \emph{no} such action that is suitable for our purposes, so instead we define an action on a certain double cover of $\mathcal{G}_{K,\ell,\ff_0}$.

\subsection{The Field of Moduli of a Cyclic $\ell^a$-isogeny}

\begin{thm}
\label{NEW6.1}
\label{THM4.1}
Let $\ell^a$ be a prime power, let $K$ be an imaginary quadratic field, and let $\varphi: E \ra E'$ be a cyclic $\ell^a$-isogeny of $K$-CM elliptic curves over $\C$, and let $\Q(\varphi)$ be the field of moduli of $\varphi$. Let $\Delta = \ell^{2 L} \ff_0^2 \Delta_K$ be the discriminant of the endomorphism ring of $E$ (here $\gcd(\ff_0,\ell) = 1$), let $\Delta' = \ell^{2 L'} \ff_0^2 \Delta_K$ be the discriminant of the endomorphism ring of $E'$, and put $\overline{L} = \max(L,L')$ and $\ff = \ell^{\overline{L}} \ff_0$.  
Then:
\begin{itemize}
\item[a)] There is a field embedding $\Q(\ff) \hookrightarrow \Q(\varphi)$. 
\item[b)] We have $\Q(\varphi) \subseteq K(\ff)$.
\end{itemize}
\end{thm}
\begin{proof}
This result is a special case of \cite[Thm. 5.1]{Clark22a} when $\ff_0^2 \Delta_K < -4$, so we may assume that $\Delta_K \in \{-3,-4\}$ and 
$\ff_0 = 1$.  
a) Certainly $\Q(\varphi)$ contains both $\Q(j(E)) \cong \Q(\ell^{L} \ff_0)$ and $\Q(j(E')) \cong \Q(\ell^{L'} \ff_0)$.  
At least one of these 
fields is isomorphic to $\Q(\ell^{\overline{L}} \ff_0) = \Q(\ff)$.  \\
b)  As usual, without loss of generality we may assume that $j(E) = j_{\Delta}$.  
Let $(E_0)_{/K(\ff)}$ be any $K$-CM elliptic curve with endomorphism ring of discriminant $\ff^2 \Delta_K$.  Since 
$\Delta \mid \ff^2 \Delta_K$, there is a canonical $K(\ff)$-rational isogeny $\varphi_0$ with source elliptic curve $E_0$ 
and whose target elliptic curve has $j$-invariant $j_{\Delta} = j(E)$.  We choose this target elliptic curve as our model 
for $E$ over $K(\ff)$, and our task is to show that for this model of $E$, the kernel of $\varphi$ is a $\gg_{K(\ff)}$-stable 
subgroup.   In fact we will show that if $\varphi$ is any cyclic $\ell^a$-isogeny with source elliptic curve $E_{/K(\ff)}$ 
and target elliptic curve of level $L'$, then $\varphi$ is defined over $K(\ff)$ in the sense that its kernel is $g_{K(\ff)}$-stable.  
  \\ \indent
The isogeny $\varphi$ decomposes into $\varphi_3 \circ \varphi_2 \circ 
\varphi_1$ with $\varphi_1: E \ra E_1$ ascending, $\varphi_2: E_1 \ra E_2$ horizontal and $\varphi_3: E_2 \ra E'$ 
descending.   We define $b,h,d \in \N$ by 
\[ \deg \varphi_1 = \ell^b, \ \deg \varphi_2 = \ell^h, \ \deg \varphi_3 = \ell^d. \]
The isogeny $\varphi_1$ is unique, so it is certainly defined over $K(\ff)$.  If $\varphi_2 \neq 1$, then $\varphi_2$ is, up to isomorphism 
on its target, given as $E_1 \ra E_1/E_1[I]$ for a nonzero $\Z_K$-ideal $I$, so $\varphi_2$ is defined over $K(j(E_1)) = K \subseteq K(\ff)$.  Thus it suffices to show that the descending $\ell^d$-isogeny $\varphi_3: E_2 \ra E'$ is defined over $K(\ff)$.  For this the more difficult case is when $E_2$ lies at the surface.  If $E_2$ lies below the surface, then whether the kernel of $\varphi_3$ is 
$\gg_{K(\ff)}$-stable is independent of the model of $E_2$, and the dual isogeny $\varphi_3^{\vee}: E' \ra E_2$ is ascending 
so is defined over $K(j(E')) = K(j_{\Delta'}) \subseteq K(\ff)$ on any model of $E'$, so $\varphi_3$ is also defined over $K(\ff)$.  
Thus we may assume that $E_2$ lies at the surface.  Since $\varphi_2: E_1 \ra E_2$ is horizontal, also $E_1$ lies at the surface.  By our 
choice of $E$, we have that $E_1$ is the target elliptic curve of a cyclic $K(\ff)$-rational $\ell^a$-isogeny with source elliptic curve $E_0$.  By \cite[Prop. 4.5]{BCI} and its proof, we have that the modulo $\ell^{\overline{L}}$-Galois representation on 
$(E_1)_{/K(\ff)}$ consists of scalar matrices, which means that every cyclic $\ell^{\overline{L}}$-isogeny on $E_1$ is 
defined over $K(\ff)$.  Since $d = L' \leq \overline{L}$, the same holds for every cyclic $\ell^d$-isogeny on $E_1$.  
If $\varphi_2 = 1$ this tells us directly that $\varphi_3$ is defined over $K(\ff)$.  In general: since $\varphi_2$ is horizontal, 
$\Z_K$ has class number $1$ and $K(\ff)$ contains $K$, then $\varphi_2$ is given, up to an isomorphism on the tareget, by a $K(\ff)$-rationally defined endomorphism of $E_2$, so $E_3$ is $K(\ff)$-rationally isomorphic to $E_2$.  It follows that every downward cyclic $\ell^b$-isogeny on 
$E_2$ has $\gg_{K(\ff)}$-stable kernel, so $\varphi_3$ is defined over $K(\ff)$.  
\end{proof}
\noindent
Thus we get a simple dichotomy for the field of moduli $\Q(\varphi)$ 
of a cyclic $\ell^a$-isogeny $\varphi$: for the specific value of $\ff$ given in Theorem \ref{NEW6.1} in terms of the endomorphism 
rings of the source and target elliptic curves of $\varphi$, we know that $\Q(\varphi)$ is isomorphic to either $\Q(\ff)$ or to $K(\ff)$.  
As in \cite[\S 5]{Clark22a}, we can resolve this dichotomy by understanding the action of complex conjugation on 
paths in the isogeny graph.

\subsection{Action of Complex Conjugation on $\mathcal{G}_{K,\ell,\ff_0}$}\label{action_of_conjugation_section}
First of all we have an action of complex conjugation --- by this we will always really mean an action of the group $\gg_R = \{1,c\}$ --- 
on the set of vertices of $\mathcal{G}_{K,\ell,\ff_0}$: indeed, the vertices are $j$-invariants of complex elliptic curves, so this is just 
obtained by restricting the natural action of $c$ on $\C$.  From \cite[\S 2.5]{Clark22a} we know that for all $L \in \Z^{\geq 0}$, the number of real vertices in level $L$ is 
\[ \rr_L \coloneqq \# \Pic \OO(\ell^{2L} \ff_0^2 \Delta_K)[2], \]
and Gauss's genus theory of binary quadratic formulas yields a formula for $\rr_L$ in terms of the number of odd prime divisors 
of $\Delta$ and the class of $\Delta$ modulo $32$\ \cite[Lemma 2.8]{Clark22a}.
\\ \\
In the absence of multiple edges, this action of $c$ on the vertex set of $\mathcal{G}_{K,\ell,\ff_0}$ determines the action on 
the graph.  When $\ff_0^2 \Delta_K < -4$ the only possible multiple edges are surface edges, on which the action of $c$ is 
easy to understand: the two nonisomorphic $\R$-structures on a real vertex differ from each other by quadratic twisting by $-1$, so the action of complex conjugation on the set of order $\ell$-subgroups is independent of the choice of $\R$-structure.  The answer 
is then that an edge running beween two real surface vertices is \emph{not} fixed by complex conjugation in the split case and is fixed by 
complex conjugation in the ramified case (there are no surface edges in the inert case). 
\\ \\
We are in the case where $\ff_0^2 \Delta_K \in \{-3,-4\}$.  Then we still have:
\\ \\
$\bullet$ If $v$ is a vertex at level $L \geq 1$ and $e: v \ra w$ is a downward edge, then it is the only edge from $v$ to $w$, so $e$ is real if and only if $v$ and $w$ are.  (Again, because we are below the surface, $\Aut E_v = \{\pm 1\}$, so the action of complex conjugation on subgroups of $E_v$ is independent of the chosen $\R$-model.)  
\\ \\
$\bullet$ An upward edge $e: v \ra w$ gets mapped under complex conjugation to the unique upward edge with initial vertex $c(v)$, 
so $e$ is real if and only if $v$ is real.  
\\ \\
The trickier cases are those of a surface edge and of an edge running from the (unique, real) surface vertex $v_0$ to a real level 
$1$ vertex.  We will discuss these in detail shortly.  
\\ \\
In general, we make use of the following convention: for all $L \in \Z^{\geq 0}$ we mark one vertex at level $L$: the one with $j$-invariant \[j_L \coloneqq j(\C/\OO(\ell^{2L} \Delta_K)). \] In our diagrams, this is always the leftmost vertex in a given level.  The lattice $\OO(\ell^{2L} \Delta_K)$ gives rise to a particular model $E_L$ over $\Q(j_{\ell^{2L} \Delta_K})$ and hence to a particular model over $\R$.  These models are compatible: for all $L \geq 1$, the upward edge from $j_L$ to $j_{L-1}$ is 
realized by the $\Q(j_L)$-rational isogeny $\C/\OO(\ell^{2L} \Delta_K) \ra \C/\OO(\ell^{2L-2} \Delta_K)$.
\\ \\
$\bullet$ Suppose $\Delta = -4$ and $\ell > 2$.  We have $\mathfrak{r}_0 = 1$ and $\mathfrak{r}_L = 2$ for all $L \geq 1$.  
Each real vertex $v$ in level $L \geq 1$ has an odd number, $\ell$, of descendant vertices, so at least one of these must be fixed by complex conjugation, and it follows that $v$ has exactly one real descendant. 
\\ \\
It remains to discuss the action of complex conjugation on the set of directed edges emanating from the surface vertex $v_0$, which corresponds to ``the'' elliptic curve $E_{/\C}$ with $j$-invariant $1728$.  By \cite[Thm. 5.3]{Clark22a}, for any real elliptic curve and any odd prime $\ell$, there are exactly $2$ order $\ell$-subgroups of $E(\C)$ stabilized by complex conjugation.  When $\ell \equiv 1 \pmod{4}$ there are two surface loops corresponding to $E[\pp]$ and $E[\overline{\pp}]$ where $\pp,\overline{\pp}$ are the two $\Z[\sqrt{-1}]$-ideals of norm $\ell$.  These two edges are interchanged by complex conjugation (independently of the chosen $\R$-structure on $E$).  
So the two real edges must be downward edges.  For each real level one vertex $v$, there is a pair of edges from $v_0$ to $v$; evidently 
complex conjugation stabilizes the pair, so if one is real, then both are real.  It follows that for exactly one of the two level $1$ real 
vertices both edges from the surface to that vertex are real, whreas for the other level $1$ real vertex neither edge is real.  Which is 
which depends upon the chosen $\R$-structure on $v_0$: indeed, indeed, for each level $1$ real vertex $v$, the unique upward 
edge $e: v \ra v_0$ can be defined over $\Q(j(E_v))$ and hence over $\R$; this provides an $\R$-model for $E$ on which the dual 
isogeny is real.  
\\ \\
If our path starts at $v_0$ and ends up at level $L$ then it is clear that the field of moduli is $K(\ell^L)$ if the path includes a 
surface edge and $\Q(\ell^L)$ otherwise.  The harder case is if our path starts at $j_L$ with $L \geq 1$ and ascends to the surface.  
In this case when we ascend to the surface we get the real model for $E$ given by the lattice $\Z[\sqrt{-1}]$, and \emph{in this real model} it is the two edges from $j_0$ to $j_1$ that are real.  The significance of this for our counting problem is that if we start below 
the surface and ascend to the surface there is a unique way to extend the path so that the corresponding isogeny is fixed under 
complex conjugation: we take the unique edge from $j_0$ to $j_1$ that is not the inverse of the ascending edge from $j_1$ to $j_0$.  
\\ \\
Thus one sees that in this case we \emph{are} able to define an action of $\gg_R$ on $\mathcal{G}_{K,\ell,1}$, but to do so we had to make a choice that was appropriate for our applications. 

\begin{figure}[H] 
\includegraphics[scale=0.5]{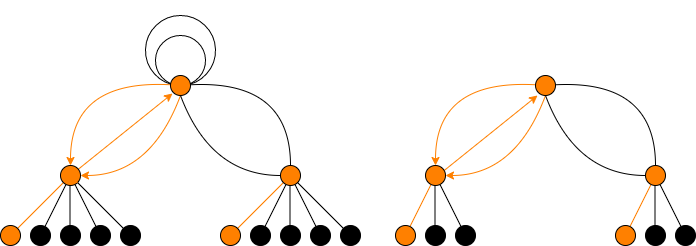}
\setlength{\abovecaptionskip}{10pt plus 0pt minus 0pt}
\caption{$\mathcal{G}_{\mathbb{Q}(\sqrt{-1}),\ell,1}$ up to level $2$ in the cases of $\ell$ split ($\ell = 5$, left) and inert ($\ell = 3$, right) in $\mathbb{Q}(\sqrt{-1})$, with vertices and edges fixed by complex conjugation colored orange}
\end{figure}

$\bullet$ Suppose $\Delta = -3$ and $\ell > 3$.  As above, we have $\mathfrak{r}_0 = 1$ and $\mathfrak{r}_L = 2$ for all $L \geq 2$.  
And again, each real vertex $v$ in level $L \geq 1$ has an odd number, $\ell$, of descendant vertices, so $v$ has a unique real descendant.  
If $\ell \equiv 1 \pmod{3}$ there is a pair of surface loops that are interchanged by complex conjugation; if $\ell \equiv 2 \pmod{3}$ 
there are no surface edges.  So by \cite[Thm. 5.3]{Clark22a} in either $\R$-model of ``the'' elliptic curve $E_{/\C}$ with $j$-invariant $0$ 
corresponding to the surface vertex $v_0$ there are precisely $2$ order $\ell$-subgroups stable under complex conjugation.  But 
this time things work out more nicely: there are three edges from $v_0$ to each of the two real level $1$ vertices, which as a set are 
stable under complex conjugation.  Since $3$ is odd, at least one edge in each set must be fixed by $c$, hence exactly one because 
there are two such edges altogether.  

\begin{figure}[H] 
\includegraphics[scale=0.5]{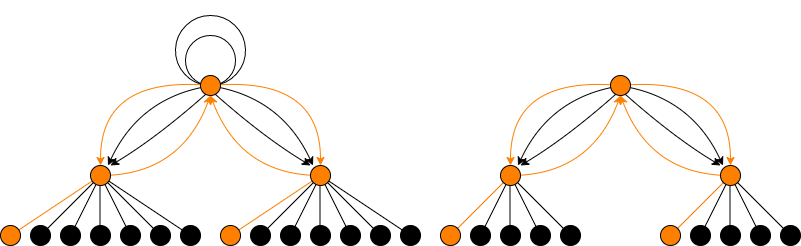}
\setlength{\abovecaptionskip}{10pt plus 0pt minus 0pt}
\caption{$\mathcal{G}_{\mathbb{Q}(\sqrt{-3}),\ell,1}$ up to level $2$ in the cases of $\ell$ split ($\ell = 7$, left) and inert ($\ell = 5$, right) in $\mathbb{Q}(\sqrt{-3})$, with vertices and edges fixed by complex conjugation colored orange}
\end{figure}

$\bullet$ Suppose $\Delta = -3$ and $\ell  =2$.  Now we have $\mathfrak{r}_0 = \mathfrak{r}_1 = 1$, $\mathfrak{r}_2 = 2$ and 
$\mathfrak{r}_L = 4$ for all $L \geq 3$.   This means that every vertex of level $L \leq 3$ is real.  For each $L \geq 3$, the real 
vertices of level $L$ can be partitioned into pairs in which each pair has a common neighbor in level $L-1$, and in each pair, 
exactly one of the two vertices has two real descendants and the other vertex has no real descendants.  This follows from the same 
argument as in the proof of \cite[Lemma 5.7c)]{Clark22a}.  

\begin{figure}[H] 
\includegraphics[scale=0.5]{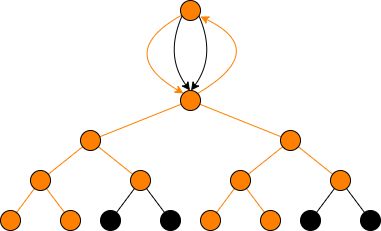}
\setlength{\abovecaptionskip}{10pt plus 0pt minus 0pt}
\caption{$\mathcal{G}_{\mathbb{Q}(\sqrt{-3}),2,1}$ up to level $4$, with vertices and edges fixed by complex conjugation colored orange}
\end{figure}

$\bullet$ Suppose $\Delta = -4$ and $\ell = 2$.  We have $\rr_0 = \rr_1 = 1$ and $\rr_L = 2$ for all $L \geq 2$.  For all $L \geq 2$, 
the vertex $v_L$ corresponding to $j$-invariant $j_L = j(\C/\OO(-2^{2L + 2})$ is real; the other real vertex in level $L$ therefore 
must be the other descendant vertex from $v_{L-1}$.  
\\ \\
Let us now discuss the action of complex conjugation on edges.  Let $E_{/\C}$ be ``the'' elliptic curve of $j$-invariant $1728$.  
In either $\R$-model of $E$, the surface loop corresponds to the isogeny with kernel $E[\pp]$, where $\pp$ is the unique prime ideal 
of $\Z[\sqrt{-1}]$ lying over $2$, which is stable under complex conjugation.  
\\ \\
If we choose the $\R$-model of $E$ with real lattice $\Z[\sqrt{-1}]$, then all three order $2$ subgroups are stable under complex conjugation: they can be seen quite clearly as $\frac{1}{2} + \Z[\sqrt{-1}$, $\frac{\sqrt{-1}}{2} + \Z[\sqrt{-1}]$ and 
$\frac{1+\sqrt{-1}}{2} + \Z[\sqrt{-1}]$.  So it may seem that we have defined an action of 
complex conjugation on $\mathcal{G}_{\Q(\sqrt{-1}),2,1}$.
\\ \indent
However this graph cannot be used for our study of isogenies!  To see why, consider either of the two paths that starts at the 
vertex $v_1$ in level $1$, ascends to level $0$, takes the surface loop, and then descends back down to level $1$.  These correspond 
to two cyclic $8$-iosgenies with source elliptic curve of discriminant $-16$.  However, contrary to what the graph suggests, neither 
of these two isogenies is defined over $\R$.  Our graph is letting us down because the 
surface loop, which can be realized on uniformizing lattices as $\C/\Z[\zeta_4] \ra \C/(1+\zeta_4)\Z[\zeta_4]$ is an isogeny 
of real elliptic curves, but the source and target have different $\R$-structures.  Recall that every elliptic curve $E_{/\C}$
with $j(E) \in \R$ has precisely two nonisomorphic $\R$-models \cite[Prop. V.2.2]{SilvermanII}.  When $j \notin \{0,1728\}$, 
these two models are just quadratic $-1$ twists of each other, but this is not the case when $j \in \{0,1728\}$.  When $j = 1728$ 
(i.e., $\Delta = -4$), for our purposes the most useful way to distinguish between the two models is to observe that in 
the model $\C/\Z[\zeta_4]$ all three order $2$ subgroups are real, whereas in the model $\C/(1+\zeta_4)\Z[\zeta_4]$ there is 
exaxctly one real order $2$ subgroup, generated by $1 + (1+\zeta_4)\Z[\zeta_4]$.  This means that in our length $3$ paths considered 
above, once we take the surface loop, we arrive at an elliptic curve over $\R$ for which the two order $2$ subgroups that correspond 
to the $2$ downward edges from $v_0$ to $v_1$ are now interchanged by complex conjugation.  
\\ \\
We remedy this by passing from $\mathcal{G} = \mathcal{G}_{\Q(\sqrt{-1}),2,1}$ to the double cover $\widetilde{\mathcal{G}}$ by unwrapping the surface loop, to get a graph that now at each level $L$, consists of two copies of the vertex set of $\mathcal{G}$ 
at level $L$.  We decree that complex conjugation acts on the second copy of the vertex set the same way it does on the first copy.  
The surface edge between the two copies of $v_0$ is real, but in the second copy the two downward edges from $v_0$ to $v_1$ 
are now complex.  Complex conjugation acts on all other edges in the second copy the same as it does in the first copy (away from the surface the action of complex conjugation on cyclic subgroups is independent of the choice of $\R$-model).  
\begin{figure}[H] 
\includegraphics[scale=0.5]{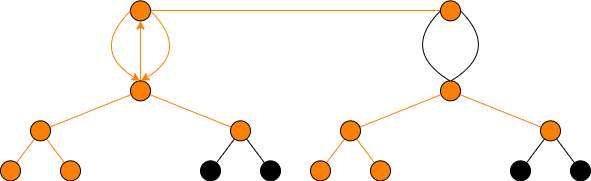}
\setlength{\abovecaptionskip}{10pt plus 0pt minus 0pt}
\caption{the double cover $\widetilde{\mathcal{G}}$ of $\mathcal{G}_{\mathbb{Q}(\sqrt{-1}),2,1}$ up to level $3$, with vertices and edges fixed by complex conjugation colored orange}
\end{figure}

\begin{remark}
\textbf{}
\begin{itemize}
\item[a)]In Figure 5, we did not draw the upward edge with initial vertex the level $1$ vertex in the right hand copy of $\mathcal{G} = \mathcal{G}_{\mathbb{Q}(\sqrt{-1}),2,1}$.  As far as the group action of $\gg_{\R}$ on $\widetilde{\mathcal{G}}$ is concerned, it 
is clear that this edge must be $c$-fixed.  However the $c$-fixedness of this edge has no elliptic curve interpretation -- no nonbacktracking path starting in the lefthand copy of $\mathcal{G}$ in $\widetilde{\mathcal{G}}$ contains this edge.  Drawing this edge as $c$-fixed 
seems to invite confusion, so we have not done so.
\item[b)] 
It's interesting to compare $\widetilde{\mathcal{G}}$ to the graph of \cite[Lemma 5.7]{Clark22a}. These graphs are not isomorphic, but their enumerations of real and complex paths are the same.
\item[c)] It is also interesting (and perhaps confusing, at first) to compare the change of real structures induced by the horizontal 
edge in $\mathcal{G}_{\Q(\sqrt{-1}),2,1}$ to the end of the proof of Theorem \ref{THM4.1}, in which the source and target curves 
of a horizontal edge are rationally isomorphic. The difference is that in the setting of Theorem \ref{THM4.1} the ground field contains $K$.
As for the horizontal edge, it corresponds to the ideal $(1+\zeta_4)$, which is real and principal...but not ``real-principal'': i.e., it is not 
generated by a real element and thus its kernel is not the kernel of an $\R$-rationally defined endomorphism.
\end{itemize}
\end{remark}
\noindent
$\bullet$ Suppose $\Delta = -3$ and $\ell = 3$.  We have $\mathfrak{r}_L = 1$ for all $L \geq 0$, so the unique real 
vertex in level $L$ is $v_L$, corresponding to the elliptic curve $\C/\OO(-3^{2L+1})$.  
\\ \\
There is a sort of ``more benign'' analogue of the phenomenon encountered in the previous case: the surface loop in this 
graph corresponds to the $\R$-isogeny $\C/\Z[\zeta_6] \ra \C/(1-\zeta_3)\Z[\zeta_6]$.  The source and target elliptic curves 
are isomorphic over $\C$ but have different $\R$-structures.  Indeed, by \cite[Lemma 3.2]{BCS17}, if $\Lambda_1$ and $\Lambda_2$ 
are real lattices in $\C$, then they determine the same $\R$-isomorphism class of elliptic curves if and only if they are \text{real homothetic}: there is $\alpha \in \R^{\times}$ such that $\Lambda_2 = \alpha \Lambda_1$.  The two lattices $\Z[\zeta_6]$ 
and $(1-\zeta_3)\Z[\zeta_6]$ are not real homothetic: one can see this directly or use \cite[Lemma 3.6a)]{BCS17}.  \\ \indent
So we defined an action of complex conjugation on the three downward edges with initial vertex the surface vertex $v_0$: one 
is real and two are complex.  After we take the surface loop we are now considering the action of complex conjugation on a 
nonisomorphic real elliptic curve.  Because of this, the principled response is to again pass from 
$\mathcal{G}_{\Q(\sqrt{-3}),3,1}$ to the double cover $\widetilde{\mathcal{G}}$ by unwrapping the surface loop to get a 
graph that at each level $L$ consists of two copies of the vertex set of $\mathcal{G}$ at level $L$, and we define the action of 
complex conjugation in the same way as above.
\begin{figure}[H] 
\includegraphics[scale=0.5]{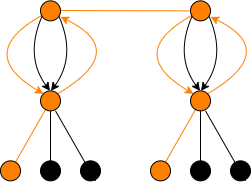}
\setlength{\abovecaptionskip}{10pt plus 0pt minus 0pt}
\caption{the double cover $\widetilde{\mathcal{G}}$ of $\mathcal{G}_{\mathbb{Q}(\sqrt{-3}),3,1}$ up to level $2$, with vertices and edges fixed by complex conjugation colored orange}
\end{figure}
\noindent
While in the previous case the change of $\R$-structure changed the \emph{number} of order $\ell = 2$ subgroups 
fixed by complex conjugation, in this case $\ell = 3$, so \cite[Thm. 5.1]{Clark22a} applies to show that in any $\R$-model 
exactly one of the three ``downward'' order $3$ subgroups is real.  So while in the previous case we needed to pass to the 
double cover in order to ensure the correctness of our enumeration of real and complex paths, in this case the enumeration of 
real and complex paths is the same whether we pass from $\mathcal{G}$ to $\widetilde{\mathcal{G}}$ or not.

\section{CM Points on $X_0(\ell^a)_{/\Q}$}\label{CM_points_prime_power_section}
Let $\ell$ be a prime number, and let $\Delta = \ell^{2L} \Delta_K$ be an imaginary quadratic discriminant with $\Delta_K \in \{-3,-4\}$.  In this section we will compute the fiber of $X_0(\ell^a) \ra X(1)$ over $J_{\Delta}$. For $\Delta < -4$ there is no ramification, 
so we determine which residue fields occur and with what multiplicity. For $\Delta \in \{-3,-4\}$, a closed point on $X_0(\ell^a)$ in the fiber over $J_{\Delta}$ has ramification, of index $2$ or $3$ in the respective cases of $\Delta = -4$ and $-3$, exactly when a path in its closed point equivalence class includes a descending edge from level $0$ to level $1$, i.e. exactly when the path is not completely horizontal. The residue field of a closed point on a finite-type $\Q$-scheme is a number field that is well-determined up to isomorphism; it is not well-defined as a subfield of $\C$.  Thus when we write that 
the residue field is $\Q(\ff)$ for some $\ff \in \Z^+$, we mean that it is isomorphic to this field.  
\\ \noindent
Without loss of generality we may take our source elliptic curve to have $j$-invariant $j_{\Delta}$. Our task is then to: \\ \\
(i) Enumerate all nonbacktracking length $a$ paths $P$ in $\mathcal{G}_{K,\ell,\ff_0}$. \\
(ii) Sort them into closed point equivalence classes $\mathcal{C}(P)$, and record the field of moduli for each equivalence 
class (we record any number field isomorphic to $\Q(f)$ as $\Q(f)$).  \\
(iii) Record how many closed point equivalence classes give rise to each field of moduli.  
\\ \noindent
In \S \ref{paths_section} we have addressed the added subtlety in the notion of backtracking when $f_0^2\Delta_K \in \{-3,-4\}$, and in \S \ref{action_of_conjugation_section} we have provided a meaningful description of the action of complex conjugation on paths in $\mathcal{G}_{K,\ell, 1}$. This provides the means to carry out our path-type analysis steps (i) through (iii), just as done in \cite[\S 7]{Clark22a} for $f_0^2\Delta_K < -4$. What we find is that the resulting enumeration of path types and corresponding residue fields for $f_0^2\Delta_K \in \{-3,-4\}$ is \emph{exactly} as in \cite[\S 7]{Clark22a} for any given $\ell$ and splitting behavior of $\ell$ in $K$, and so we refer the reader to the enumeration provided therein. 
\\ \\
A check on the accuracy of our calculations is as follows: let $\psi: \Z^+ \ra \Z^+$ be the multiplicative function such that 
for any prime power $\ell^a$ we have $\psi(\ell^a) = \ell^{a-1}(\ell+1)$.  For all $N \in \Z^+$, we have (e.g. \cite[Lemma 4.1a)]{CGPS22})
\[ \deg(X_0(N) \ra X(1)) = \psi(N). \]
Letting $d_{\varphi}$ and $e_{\varphi}$ denote, respectively, the residual degree and ramification index of the closed point $[\varphi]$ with respect to the map $X_0(\ell^a) \rightarrow X(1)$, we must have 
\[ \sum_{C(\varphi)} d_{\varphi} \cdot e_{\varphi} = \psi(\ell^a) = \ell^a + \ell^{a-1}, \]
where the sum extends over closed point equivalence classes of points in the fiber over $J_{\Delta}$.

\section{The Projective Torsion Field}
\noindent
Let $F$ be a field of characteristic $0$, let $E_{/F}$ be an elliptic curve, and let $N \in \Z^{\geq 2}$.  The \textbf{projective $N$-torsion field} 
$F( \PP E[N])$ is the kernel of the modulo $N$ projective Galois representation, i.e., the composite homomorphism 
\[ \overline{\rho_N}: \gg_F \stackrel{\rho_N}{\ra} \Aut E[N] \ra \Aut \PP E[N], \]
where $\PP E[N]$ denotes the projectivization of the $2$-dimensional $\Z/N\Z$-module $E[N](\overline{F})$.  After choosing a $\Z/N\Z$-basis 
for $E[N]$, we may view $\overline{\rho_N}$ as a homomorphism from $\gg_F$ to $\PGL_2(\Z/N\Z)$.   Thus $F(\P E[N])/F$ is a finite degree 
Galois extension.  The projective $N$-torsion field of $E/F$ is also characterized as the minimal algebraic extension of $F$ over which all cyclic $N$-isogenies with source elliptic 
curve $E$ are defined.
\\ \\
The following result is a small refinement of \cite[Prop. 4.5]{BCI}.

\begin{prop}
\label{REFINEDBC4.5}
Let $\Delta = \ff^2 \Delta_K$ be an imaginary quadratic discriminant, let $N \geq 2$, and let $E_{/K(N\ff)}$ be a $\Delta$-CM elliptic curve.
\begin{itemize}
\item[a)] The following are equivalent:
\begin{itemize}
\item[(i)] We have $K(N\ff)( \PP E[N]) = K(N\ff)$. 
\item[(ii)] There is a $K(N\ff)$-rational cyclic $N$-isogeny $\varphi: E \ra E'$, where $E'$ is an $N^2 \Delta$-CM elliptic curve.
\end{itemize}
\item[b)] For every $\Delta$-CM elliptic curve $E_{/K(N\ff)}$, there is an elliptic curve $E_0{/K(N\ff)}$ with $j(E_0) = j(E)$ and 
such that $E_0$ satisfies the equivalent conditions of part a).  Moreover, an elliptic curve $E'_{/K(N\ff)}$ with $j(E') = j(E)$ 
satisfies the equivalent conditions of part a) if and only if $E'$ is a quadratic twist of $E_0$.  
\end{itemize}
\end{prop}
\begin{proof}
As usual, it is no loss of generality to assume that $j(E) = j_{\Delta}$.  \\
a) The implication (ii) $\implies$ (i) follows from \cite[Prop. 4.5]{BCI} and its proof.  As for (i) $\implies$ (ii), we may take $\varphi$ to 
be the dual of the isogeny $\psi: \C/\OO(N^2 \Delta) \ra \C/\OO(\Delta)$, which because of (i) must be $K(N\ff)$-rational on $E$. \\
b) The isogeny $\psi$ is defined over $\Q(N\ff)$, hence also over $K(N\ff)$.  Since $N^2 \Delta < -4$, the $K(N\ff)$-rational 
model of an elliptic curve with $j$-invariant $j_{N^2 \Delta}$ is unique up to quadratic twist.  If $F$ is a field of characteristic different 
from $2$, $\psi: E_1 \ra E_2$ is an $F$-rational isogeny with kernel $C$, and $d \in F^{\times}/F^{\times 2}$, then $C$ 
remains $F$-rational on the quadratic twist $E_1^d$ and we have $E_1^d/C \cong_F E_2^d$.  This shows that the elliptic curve 
$E_0$ of part b) exists and is unique up to quadratic twist; finally, quadratic twists do not change rationality of isogenies hence 
do not change projective torsion fields. 
\end{proof}
\noindent
For an imaginary quadratic discriminant $\Delta$, let $\OO(\Delta)$ be the imaginary quadratic order of discriminant $\Delta$ and let 
\[ w_{\Delta} \coloneqq \# \OO(\Delta)^{\times}. \]

\begin{thm}
\label{BIGPROJTHM}
Let $\Delta = \ff^2 \Delta_K$ be an imaginary quadratic discriminant, and let $N \geq 3$.
\begin{itemize}
\item[a)] Let $P \in X_0(N,N)$ be a $\Delta$-CM closed point.  Then $\Q(P) = K(N\ff)$.
\item[b)] Let $F$ be a field of characteristic $0$, and let $E_{/F}$ be a $\Delta$-CM elliptic curve.  
Then $F(\P E[N]) \supseteq K(N\ff)$ and $[F(\P E[N]):F K(N\ff)] \mid \frac{ \# w_{\Delta}}{2}$.
\end{itemize}
\end{thm}
\begin{proof}Again we may assume without loss of generality that $j(P) = j_{\Delta}$. \\
a) By Proposition \ref{REFINEDBC4.5}, there is a $\Delta$-CM elliptic curve $E_{/K(N\ff)}$ on which the projective modulo $N$ Galois representation 
is trivial.  This elliptic curve induces a $\Delta$-CM closed point $P_0 \in X_0(N,N)$ such that $\Q(P_0)$ can be embedded into $K(N\ff)$.  
Moreover, in the notation of \cite[\S 1.1]{Clark22a}, the subgroup $H_0(N,N)$ of $\GL_2(\Z/N\Z)/\{\pm 1\}$ used to define the modular curve $X_0(N,N)$ is the subgroup of scalar matrices, which is normal, hence $X_0(N,N) \ra X(1)$ is a Galois covering of curves over $\Q$.  It follows that all residue fields of closed points of $X_0(N,N)$ lying over the closed point $J_{\Delta}$ of $X(1)$ are isomorphic.  It 
follows that $\Q(P)$ is isomorphic to a subfield of $K(N\ff)$.  By \cite[Prop. VI.3.2]{Deligne-Rapoport73} there is an elliptic curve 
$E_{/\Q(P)}$ inducing $P$ with trivial projective modulo $N$ Galois representation.  As in the proof of Proposition \ref{REFINEDBC4.5} 
we have a $\Q(P)$-rational isogeny $\varphi: E \ra E'$ with $j(E') = j_{N^2 \Delta}$, so $\Q(P)$ contains $\Q(N \ff)$.   Since $[K(N\ff):\Q(N\ff)] = 2$, we have 
either $\Q(P) = \Q(N\ff)$ or $\Q(P) = K(N\ff)$.  However, if $\Q(P) = \Q(N\ff)$, then since $\Q(N\ff) \subset \R$, we get a real elliiptic curve with 
real projective $N$-torsion field, contradicting \cite[Cor. 5.4]{Clark22a}.  \\
b) From part a) we know that $F(\P E[N]) \supseteq K(N\ff)$.  Consider the base 
extension of $E$ to $L \coloneqq FK(N\ff)$.  It follows from part a) that there is a character $\chi: \gg_L \ra \mu_{w_{\Delta}}$ 
such that the twist of $E_{/L}$ by $\chi$ has trivial projective mod $N$ Galois representation.   There is then a cyclic field extension $M/L$ of degree dividing $\frac{w_{\Delta}}{2}$ such that 
\[ \chi(\gg_M) \subset \{ \pm 1\}, \]
which means that there is a quadratic twist of $E_{/M}$ for which the projective mod $N$ Galois representation is trivial.  But quadratic 
twists do not affect the projective modulo $N$ Galois representation, so the projective Galois representation on $E_{/M}$ is 
trivial, and $[M:FK(N\ff)] \mid \frac{w_{\Delta}}{2}$.  
\end{proof}
\noindent
When $N = 2$, the projective $N$-torsion field of an elliptic curve $E_{/F}$ is just its $2$-torsion field $F(E[2])$.  Because of this the 
$N = 2$ analogue of Theorem \ref{BIGPROJTHM} had already been known, but for future reference we record the result anyway.

\begin{prop}
\label{Proj_Remark_leq2} 
\label{PROP6.2}
Let $\Delta = \ff^2 \Delta_K$ be an imaginary quadratic discriminant, let $F$ be a field of characteristic $0$, and let $E_{/F}$ be a $\Delta$-CM elliptic curve.  Let $P \in X_0(2,2)_{/\Q}$ be a $\Delta$-CM point.  Then:
\begin{itemize}
\item[a)] If $\Delta < -4$ is odd, then $\Q(P) = K(2\ff)$. 
\item[b)] If $\Delta < -4$ is even, then $\Q(P) \cong \Q(2\ff)$.
\item[c)] If $\Delta = -4$, then $\Q(P) = \Q = \Q(2\ff)$.
\item[d)] If $\Delta = -3$, then $\Q(P) = K = K(2\ff)$.  
\end{itemize}
\end{prop}
\begin{proof} Again, because $X_0(2,2) = X(2) \ra X(1)$ is a Galois covering, all the residue fields of $P$ on $X_0(2,2)$ lying over 
the closed point $J_{\Delta}$ of $X(1)$ are isomorphic.  \\
a),b) Suppose $\Delta < -4$.   The results follow from \cite[Thm. 4.2]{BCS17} together with the observation that there is a point $P$ of order $2$ on $E(\overline{F})$ such that $E/\langle P \rangle$ has $4\Delta$-CM.  (They can also be obtained from an analysis of $2$-isogeny graphs.)  \\
c) The $-4$-CM elliptic curve $E: y^2 = x^3-x$ has $\Q(E[2]) = \Q$. \\
d) For all $B \in \Q^{\times}$, the curve 
\[ E_B: y^2 = x^3 + B \]
is a $-3$-CM elliptic curve with $\Q([E])2) = \Q(\zeta_3,B^{1/3})$.  Each of these fields contains $\Q(\zeta_3) = K$, so 
$\Q(P) \supseteq K$. Moreover $\Q(E_{1}) = K$, so $\Q(P) = K$.   
\end{proof}

\section{Primitive Residue Fields of CM points on $X_0(\ell^{a'},\ell^a)$}
\noindent
Let $\Delta = \ff^2 \Delta_K$ be an imaginary quadratic discriminant, let $\ell$ be a prime and let $0 \leq a' \leq a$ be integers. In this section we extend the work of \cite[\S 7]{Clark22a}, determining all primitive residue fields and degrees of $\Delta$-CM points on $X_0(\ell{a'}, \ell^a)$, to the cases in which $\Delta_K \in \{-3,-4\}$.
\\ \\
It is no loss of generality to assume that the $j$-invariant of our $\Delta$-CM elliptic curve is $j_{\Delta}$, and we shall do so throughout 
this section.
\\ \\
For $X(H)_{/\Q}$ a modular curve, we call the residue field $\mathbb{Q}(P)$ of a closed $\Delta$-CM point $P \in X(H)$ a \textbf{primitive residue field} of $\Delta$-CM points on $X(H)$ if there is no other $\Delta$-CM point $Q \in X(H)$ together with an embedding of the residue field $\Q(Q)$ into $\Q(P)$ as a proper subfield. We call the degree $d = [\mathbb{Q}(P) : \mathbb{Q}]$ a \textbf{primitive degree} of $\Delta$-CM points on $X(H)$ if there is no $\Delta$-CM point $Q \in X(H)$ such that $[\mathbb{Q}(Q) : \mathbb{Q}]$ properly divides $d$. 
\\ \\
Throughout this section, when working with a $\Delta = \ff^2\Delta_K$-CM point we put 
\[L \coloneqq \ord_{\ell}(\ff). \]

\subsection{$X_0(\ell^a)$}
In the case of $a' = 0$, i.e., of $X_0(\ell^a)$, our results of \S \ref{CM_points_prime_power_section} imply immediately that the case analysis is exactly as in \cite[\S 8.1]{Clark22a}. We recall the answer here for completeness:
\\ \\
\textbf{Case 1.1}: Suppose $\ell^a = 2$. \\
\textbf{Case 1.1a}: Suppose  $\left(\frac{\Delta}{2} \right) \neq -1$.  The primitive residue field is $\Q(\ff)$ (which equals $\Q(2\ff)$ 
when $\left( \frac{\Delta}{2} \right)  = 1$).  \\
\textbf{Case 1.1b}: Suppose $\left( \frac{\Delta}{2} \right) = -1$.  The primitive residue field is $\Q(2\ff)$.  
\\  
\textbf{Case 1.2}: Suppose $\ell^a > 2$ and $\left( \frac{\Delta}{\ell} \right) = 1$.  The primitive residue fields 
are $\Q(\ell^a \ff)$ and $K(\ff)$.  
\\ 
\textbf{Case 1.3}: Suppose $\ell^a > 2$ and $\left( \frac{\Delta}{\ell} \right) = -1$.  The primitive residue field is $\Q(\ell^a \ff)$.  
\\ 
\textbf{Case 1.4}: Suppose $\ell^a > 2$, $\left(\frac{\Delta}{\ell}\right) = 0$ and $L = 0$.  The primitive residue 
field is $\Q(\ell^{a-1} \ff)$.  
\\  
\textbf{Case 1.5}: Suppose $\ell > 2$, $L \geq 1$ and $\left( \frac{\Delta_K}{\ell} \right) = 1$.  \\
\textbf{Case 1.5a}: Suppose $a \leq 2L$.  In this case there is a $\Q(\ff)$-rational cyclic $\ell^a$-isogeny, so the only primitive 
residue field is $\Q(\ff)$.  
\\
\textbf{Case 1.5b}: Suppose $a > 2L$.  Then the primitive residue fields are $\Q(\ell^{a-2L} \ff)$ and $K(\ff)$.  
\\ 
\textbf{Case 1.6}: Suppose $\ell > 2$, $L \geq 1$, and $\left( \frac{\Delta_K}{\ell} \right) = -1$. \\
\textbf{Case 1.6a}: Suppose $a \leq 2L$.  As in Case 1.5a, there is a $\Q(\ff)$-rational cyclic $\ell^a$-isogeny, so the only primitive 
residue field is $\Q(\ff)$.  
\\
\textbf{Case 1.6b}: Suppose $a > 2L$.  In this case the primitive residue field is $\Q(\ell^{a-2L} \ff)$.  
\\ 
\textbf{Case 1.7}: Suppose $\ell > 2$, $L \geq 1$, $\left( \frac{\Delta_K}{\ell} \right) = 0$. \\
\textbf{Case 1.7a}: Suppose $a \leq 2L+1$. As in Case 1.5a, there is a $\Q(\ff)$-rational cyclic $\ell^a$-isogeny, so the only primitive 
residue field is $\Q(\ff)$.  \\
\textbf{Case 1.7b}: Suppose $a \geq 2L+2$.  In this case the primitive residue field is $\Q(\ell^{a-2L-1} \ff)$.  
\\  
\textbf{Case 1.8}: Suppose $\ell = 2$, $a \geq 2$, $L \geq 1$, and $\left( \frac{\Delta_K}{2} \right) = 1$. \\
\textbf{Case 1.8a}: Suppose $L = 1$.  The primitive residue fields are $\Q(2^a \ff)$ and $K(\ff)$.  
\\
\textbf{Case 1.8b}: Suppose $L \geq 2$ and $a \leq 2L-2$.  The primitive residue field is $\Q(\ff)$.  
\\
\textbf{Case 1.8c}: Suppose $L \geq 2$ and $a \geq 2L-1$.  The primitive residue fields are $\Q(2^{a-2L+2} \ff)$ and $K(\ff)$.  
\\  
\textbf{Case 1.9}: Suppose $\ell = 2$, $a \geq 2$, $L \geq 1$, and $\left( \frac{\Delta_K}{2} \right) = -1$. \\
\textbf{Case 1.9a}: Suppose $L = 1$.  The primitive residue fields are $\Q(2^a \ff)$ and $K(2^{a-2} \ff)$.  
\\
\textbf{Case 1.9b}: Suppose $L \geq 2$ and $a \leq 2L-2$.  The primitive residue field is $\Q(\ff)$.  
\\
\textbf{Case 1.9c}: Suppose $L \geq 2$ and $a \geq 2L-1$.  The primitive residue fields are $\Q(2^{a-2L+2} \ff)$ and $K(2^{\max(a-2L,0)}\ff)$.  
\\ 
\textbf{Case 1.10}: Suppose $\ell = 2$, $a \geq 2$, $L \geq 1$, $\left( \frac{\Delta_K}{2} \right) = 0$, and $\ord_2(\Delta_K) = 2$.  
\\
\textbf{Case 1.10a}: Suppose $a \leq 2L$.  The primitive residue field is $\Q(\ff)$.  
\\
\textbf{Case 1.10b}: Suppose $a \geq 2L+1$.  The primitive residue fields are $\Q(2^{a-2L} \ff)$ and $K(2^{a-2L-1} \ff)$.  
\\  
\textbf{Case 1.11}: Suppose $\ell = 2$, $a \geq 2$, $L \geq 1$,  $\left( \frac{\Delta_K}{2} \right) = 0$, and $\ord_2(\Delta_K) = 3$.  
\\
\textbf{Case 1.11a}: Suppose $a \leq 2L+1$.  The primitive residue field is $\Q(\ff)$.  
\\
\textbf{Case 1.11b}: Suppose $a \geq 2L+2$.  The primitive residue field is $\Q(2^{a-2L-1} \ff)$.  

\subsection{A field of moduli result} Now suppose that $1 \leq a' \leq a$ are integers, and let 
 $P \in X_0(\ell^{a'},\ell^a)$ be a 
closed $\Delta$-CM point. Then we have morphisms 
\[ \alpha: X_0(\ell^{a'},\ell^a) \ra X_0(\ell^{a'},\ell^{a'}) \text{ and } \beta: X_0(\ell^{a'},\ell^a) \ra X_0(\ell^a). \]

\begin{thm}
\label{FOMTHM}
Let $\ell$ be a prime number, let $1 \leq a' \leq a$ be positive integers and let $P \in X_0(\ell^{a'},\ell^a)$ be a closed $\Delta$-CM point.  
\begin{itemize}
\item[a)] We have 
\[ \Q( \alpha(P),\beta(P)) \subseteq \Q(P) \subseteq K(\alpha(P),\beta(P)). \]
\item[b)] We have $\Q(P) = \Q(\alpha(P),\beta(P))$ if \emph{any} of the following conditions holds:
\begin{itemize}
\item[(i)] $\Delta \neq -4$. 
\item[(ii)] $\ell^{a'} \geq 3$. 
\item[(iii)] $a = 1$.
\end{itemize}
\end{itemize}
\end{thm}
\begin{proof}
Since $\Q(P)$ is an extension of both $\Q(\alpha(P))$ and $\Q(\beta(P))$, clearly \[\Q(P) \supseteq \Q(\alpha(P),\beta(P)). \]
If $\Delta < -4$, then conversely $\Q(P) \subseteq \Q(\alpha(P),\beta(P))$: indeed, in this case, the projective torsion field 
is independent of the model.   So we may suppose $\Delta \in \{-3,-4\}$. \\
Step 2: We will show that 
\begin{equation}
\label{FOMTHMEQ}
 K(\alpha(P),\beta(P)) = K(\ell^{a'} \ff)K(\beta(P)). 
\end{equation}
Since $\Delta \in \{-3,-4\}$, the point $\beta(P) \in X_0(\ell^a)$ corresponds to a path of length $a$ 
with initial vertex at the surface; if the terminal vertex has level $L'$, then $K(\beta(P)) = K(\ell^{L'})$, so if we put 
\[ \overline{L} \coloneqq \max(a',L'), \]
then 
\[K(\alpha(P),\beta(P)) = K(\ell^{\overline{L}}). \]
We may factor an isogeny inducing $\beta(P)$ as 
$\varphi_d \circ \varphi_h$ where $\varphi_h$ is horizontal and $\varphi_d$ is descending of length $L'$.  By the results of \S 6, there is an 
elliptic curve $E_{/K(\ell^{\overline{L}})}$ with $j$-invariant $j_{\Delta}$ on which the modulo $\ell^{\overline{L}}$ Galois representation 
is given by scalar matrices.  The point $\beta(P)$ is induced by a cyclic $\ell^a$-isogeny of $\C$-elliptic curves $\varphi: E \ra E'$; 
since $\Delta = \Delta_K$, this isogeny factors as $\varphi_d \circ \varphi_h$, where $\varphi_h: E \ra E''$ is horizontal of degree $\ell^{a-L'}$ and $\varphi_d: E'' \ra E'$ is descending of degree $\ell^{L'}$.  We claim that every isogeny of this form is defined over $K(\ell^{\overline{L}})$. Indeed, as in the proof of Theorem \ref{THM4.1} we have that $\varphi_h$ is defined over $K(\ell^{\overline{L}})$ 
and moreover $E'' \cong_{K(\ell^{\overline{L}})} E$.  It follows that the modulo $\ell^{L'}$-Galois representation on $E''$ is given by 
scalar matrices, so $\varphi_d$ is also defined over $K(\ell^{\overline{L}})$ and thus $\varphi$ is as well.  \\
Step 3:  By Theorem \ref{BIGPROJTHM} and Proposition \ref{PROP6.2} we have that $\Q(\alpha(P)) = K(\ell^{a'} \ff)$ if either $\ell^{a'} \geq 3$ or $\Delta$ is odd.  Thus in either of these cases we have 
\begin{equation}
\label{FOMCONTAINSKEQ}
 \Q(P) = K(\ell^{a'} \ff) \Q(\beta(P)).  
\end{equation}
Step 4: Finally, if $a = 1$ then $\alpha: X_0(2,2) \ra X_0(2)$ is the identity map, so $\Q(P) = \Q(\alpha(P))$ holds trivially.
\end{proof}
\noindent

\begin{cor}
\label{FOMCOR}
Let $a \geq 2$, and let $P \in X_0(2,2^a)$ be a $-4$-CM point.  Let $\varphi: E \ra E'$ be an isogeny of complex elliptic curves 
inducing $\beta(P) \in X_0(2^a)$.
\begin{itemize}
\item[a)] If $\varphi$ is purely descending, then $\Q(P) = \Q(\beta(P)) \cong \Q(2^a)$.  
\item[b)] Otherwise, $\Q(P) = K(\beta(P)) = K(2^{a-1})$.
\end{itemize}
\end{cor}
\begin{proof}  Since $j_E = 1728$, there are precisely two possibilities for $\varphi$: it either consists 
of $a$ descending edges, or it has one horizontal edge followed by $a-1$ descending edges. \\ \indent In the former case, we have 
$\Q(\varphi) \cong \Q(2^a)$.  Moreover, on the model of $E$ that makes $\varphi$ $\Q(\varphi)$-rational, we evidently have a 
descending $\Q(\varphi)$-rational $2$-isogeny.  As for any $-4$-CM elliptic curve defined over $\Q(\varphi)$, we have a horizontal 
$\Q(\varphi)$-rational $2$-isogeny.  Thus the $\gg_{\Q(\varphi)}$-action on the three order $2$ subgroup schemes of $E$ fixes 
two of the subgroups, so it must also fix the third.  We conclude that $\Q(P) = \Q(\beta(P))$ in this case. \\ \indent
Now suppose that $\varphi$ consists of a horizontal edge followed by $a-1 \geq 1$ descending edges.  Now $\Q(\varphi) \cong \Q(2^{a-1})$, which is real number field, so if $\Q(P) = \Q(\beta(P))$ then we would have $\Q(P) \cong \Q(2^{a-1})$, a real number field.  But 
as we saw in \S 4.2, the horizontal $2$-isogeny $\iota: E \ra E'$ on a real ellliptic curve with $j$-invariant $1728$ interchanges the two $\R$-structures on this elliptic curve, and on precisely one of the two $\R$-structures do we have an $\R$-rational descending $2$-isogeny.  
If $\Q(P) \subseteq \R$ then we would have $\R$-rational descending $2$-isogenies defined on both the source and target of $\iota$, 
which is not possible.  Therefore in this case $\Q(P) \supseteq \Q(\beta(P))$, so by Theorem \ref{FOMTHM} we have $\Q(P) = K(\beta(P)) = 
K(2^{a-1})$.  
\end{proof}

\subsection{$X_0(\ell^{a'},\ell^a)$}
Using Theorem \ref{FOMTHM}, Corollary \ref{FOMCOR} and the work of \S 5 and \S 6, it is easy to compute all primitive residue fields $\Q(P)$ of $\Delta$-CM 
points $P \in X_0(\ell^{a'},\ell^a)$.  Indeed:
\\ \\
$\bullet$ Suppose $\ell^{a'} \geq 3$.  Then for any $\Delta$-CM point $P \in X_0(\ell^{a'},\ell^a)$, equation (\ref{FOMCONTAINSKEQ}) applies.  So if $L'$ is the minimal level such that there is a nonbacktracking path in $\mathcal{G}_{K,\ell,\ff_0}$ 
starting in level $L$ (where $\Delta = \ell^{2L} \ff_0^2 \Delta_K$), the unique primitive residue field is $K(\ell^{\max(a',L')})$.  So: 
\\ \\
\textbf{Case 2.1}: If $\left( \frac{\Delta_K}{\ell} \right) = 1$, the primitive residue field is $K(\ell^{a'} \ff)$.  \\
\textbf{Case 2.2}: If $\left( \frac{\Delta_K}{\ell} \right) = -1$, the primitive residue field is $K(\ell^{\max(a',a-2L)} \ff)$.  \\
\textbf{Case 2.3}: If $\left( \frac{\Delta_K}{\ell} \right) = 0$, the primitive residue field is $K(\ell^{\max(a',a-2L-1)} \ff)$.  
\\ \\
$\bullet$ Suppose $\ell^{a'} = 2$ and $\Delta$ is odd.  Again equation (\ref{FOMCONTAINSKEQ}) applies and the unique 
primitive residue field is $K(2^{\max(a',L')}\ff) = K(2^{\max(1,L')}\ff)$.  So:
\\ \\
\textbf{Case 3.1}: If $a = 1$, the primitive residue field is $K(2\ff)$.  \\
\textbf{Case 3.2}: If $a \geq 2$ and $\left( \frac{\Delta}{2} \right) = 1$, the primitive residue field is $K(2\ff) = K(\ff)$.  \\
\textbf{Case 3.3}: If $a \geq 2$ and $\left( \frac{\Delta}{2} \right) = -1$, the primitive residue field is $K(2^a \ff)$.  
\\ \\
$\bullet$ Suppose $\ell^{a'} = 2$ and $\Delta$ is even. For a $\Delta$-CM point $P \in X_0(2,2^a)$ Theorem \ref{FOMTHM} and Corollary \ref{FOMCOR} tell us that if $\Delta \neq -4$, if $a=1$ or if an isogeny $\varphi$ inducing $\beta(P)$ is purely descending then 
\[\mathbb{Q}(P) = \mathbb{Q}(\alpha(P),\beta(P)), \] 
and otherwise we have $\mathbb{Q}(P) = K(2^{a-1}) = K(2^{a-1}\ff)$.
\\ \\
Our casework for $\Delta$-CM points on $X_0(2,2^a)$ is then as follows:
\\ \\
\textbf{Case 4.0:} $a = 1$.  The primitive residue field is $\Q(2\ff)$.
\\ \\
\textbf{Case 4.1:} $a \geq 2$, $L = 0$ and $\ord_2(\Delta_K) = 2$.  The primitive residue fields are $\Q(2^a \ff)$ and $K(2^{a-1} \ff)$.
\\ 
\textbf{Case 4.2:} $a \geq 2$, $L = 0$ and $\ord_2(\Delta_K) = 3$.  The primitive residue field is $\Q(2^{a-1} \ff)$.
\\ 
\textbf{Case 4.3:} $a \geq 2$, $L= 1$ and $\left( \frac{\Delta_K}{2} \right) = 1$. The primitive residue fields are $\Q(2^a \ff)$ and $K(2\ff)$.
\\ 
\textbf{Case 4.4:} $\left( \frac{\Delta_K}{2} \right) = 1$, $L \geq 2$ and $2 \leq a \leq 2L-1$.  The primitive residue field is $\Q(2\ff)$.
\\ 
\textbf{Case 4.5:} $\left( \frac{\Delta_K}{2} \right) = 1$, $L \geq 2$ and $a \geq 2L$.  The primitive residue fields 
are $\Q(2^{a-2L+2} \ff)$ and $K(2\ff)$. 
\\ 
\textbf{Case 4.6:} $\left( \frac{\Delta_K}{2} \right) = -1$, $L = 1$ and $a = 2$.  The primitive residue fields are $\Q(2^2 \ff)$ 
and $K(2 \ff)$.
\\ 
\textbf{Case 4.7:} $\left( \frac{\Delta_K}{2} \right) = -1$, $L = 1$ and $a \geq 3$.  The primitive residue 
fields are $\Q(2^a \ff)$ and $K(2^{a-2} \ff)$.
\\ 
\textbf{Case 4.8}:  $\left( \frac{\Delta_K}{2} \right) = -1$, $L \geq 2$ and $2 \leq a \leq 2L-1$.  The primitive residue field is $\Q(2\ff)$.
\\ 
\textbf{Case 4.9}:  $\left( \frac{\Delta_K}{2} \right) = -1$, $L \geq 2$ and $a = 2L$.  The primitive residue fields are 
$\Q(2^2 \ff)$ and $K(2\ff)$.
\\ 
\textbf{Case 4.10}:   $\left( \frac{\Delta_K}{2} \right) = -1$, $L \geq 2$ and $a \geq 2L+1$.  The primitive residue fields are 
$\Q(2^{a-2L+2} \ff)$ and $K(2^{a-2L} \ff)$.
\\ 
\textbf{Case 4.11}:  $\ord_2(\Delta_K) = 2$, $L \geq 1$ and $2 \leq a \leq 2L+1$.  The primitive residue field is $\Q(2\ff)$.
\\ 
\textbf{Case 4.12}:  $\ord_2(\Delta_K) = 2$, $L \geq 1$ and $a \geq 2L+2$.  The primitive residue fields are
$\Q(2^{a-2L} \ff)$ and $K(2^{a-2L-1} \ff)$.  
\\
\textbf{Case 4.13}:  $\ord_2(\Delta_K) = 3$, $L \geq 1$ and $2 \leq a \leq 2L+1$.  The primitive residue field is $\Q(2\ff)$.
\\ 
\textbf{Case 4.14}:  $\ord_2(\Delta_K) = 3$, $L \geq 1$ and $a \geq 2L+2$.  The primitive residue field is $\Q(2^{a-2L-1}\ff)$.  

\section{CM points on $X_0(M,N)_{/\Q}$}
\noindent
Throughout this section $\Delta = \ff^2 \Delta_K$ is an imaginary quadratic discriminant with $\Delta_K \in \{-3,-4\}$, and $M \mid N$ are positive integers. We now discuss how to use our work developed thus far to determine the $\Delta$-CM locus on $X_0(M,N)_{/\Q}$. In \S \ref{compile_section1} we recall how the compiling across prime powers process works for $\Delta < -4$, and in \S \ref{compile_section2} we provide a result for compiling across prime powers for $\Delta \in \{-3,-4\}$. In the remainder of this section, we give an explicit description of all primitive residue fields and primitive degrees in this case.

\subsection{Compiling Across Prime Powers with $\Delta < -4$}\label{compile_section1}
For this section, we suppose $\Delta< -4$ with $\Delta_K \in \{-3,-4\}$. With this assumption, \cite[Prop. 3.5]{Clark22a} applies and our compiling across prime powers process works much the same as in \cite[\S 9.1]{Clark22a}. We elaborate here for completeness of our discussion. 

For a prime $\ell$ and integers $0 \leq a' \leq a$, the fiber $F$ of 
$X_0(\ell^{a'},\ell^a) \ra X(1)$ over the closed point $J_{\Delta}$ is a finite \'etale $\Q(J_{\Delta})$-scheme, i.e. is isomorphic to a product of finite degree field extensions of $\Q(\ff)$. Our work up to now shows that the residue field of any CM point on $X_0(\ell^{a'},\ell^{a})$ is either a ring class field or a rational ring class field, and so there are non-negative integers $b_0,\ldots,b_a,c_1,\ldots,c_{a}$ 
such that $F \cong \Spec A$, where 
\begin{equation}
\label{COMPILEQ1}
 A = \prod_{j=0}^a \Q(\ell^j \ff)^{b_j} \times \prod_{k=0}^{a} K(\ell^k \ff)^{c_k}.
\end{equation}
When $a' = 0$, the explicit values of the $b_j$'s and $c_k$'s can be determined from our results in \S \ref{CM_points_prime_power_section}. When $\ell^{a'} \geq 3$ or $\Delta$ is odd, by Theorem \ref{BIGPROJTHM} we have $b_j = 0$ for all $0 \leq j \leq a$.  
\\ \\
We now explain how the previous results allow us to compute the fiber $F = \Spec A$ of $X_0(M,N) \ra X(1)$ over $J_{\Delta}$ for any positive integers $M \mid N$, where $M = \ell_1^{a_1'} \cdots \ell_r^{a_r'}$ and $N = \ell_1^{a_1} \cdots \ell_r^{a_r}$.  For $1 \leq i \leq r$, let $F_i \cong \Spec A_i$ be the fiber of $X_0(\ell_i^{a_i'},\ell_i^{a_i}) \ra X(1)$ over $J_{\Delta}$.  By \cite[Prop 3.5]{Clark22a} we have 
\begin{equation}
\label{COMPILEQ2}
 A \cong A_1 \otimes_{\Q(J_{\Delta})} \cdots \otimes_{\Q(J_{\Delta})} A_r. 
\end{equation}
It follows that $A$ is isomorphic to a direct sum of terms of the form 
\[ B \coloneqq B_1 \otimes_{\Q(\ff)} \cdots \otimes_{\Q(\ff)} B_r, \]
where for $1 \leq i \leq r$ we have that $B_i$ is isomorphic to either $\Q(\ell_i^{j_i} \ff)$ for some $0 \leq j_i \leq a$ or to $K(\ell^{j_i} \ff)$ for some $0 \leq j_i \leq a$. 

Let $s$ be the number of indices $1 \leq i \leq r$ such that $K$ is contained in $B_i$, i.e. such that $B_i \cong K(\ell_i^{j_i})$. Because $\ff > 1$, Proposition \ref{SAIAPROP2} gives:f
\[ B \cong \begin{cases} \mathbb{Q}(\ell_1^{j_1} \cdots \ell_r^{j_r} \ff) \quad &\text{ if } s = 0 \\
K(\ell_1^{j_1} \cdots \ell_r^{j_r} \ff)^{2^{s-1}} \quad &\text{ otherwise}. \end{cases} \]
(Note that $\ell_i^{2j_i} \Delta \in \{-12,-16,-27\}$ can only occur if $j_i = 0$, due to our $\ff > 1$ assumption.) We therefore reach the following extension of \cite[Theorem 9.1]{Clark22a}:

\begin{thm}
\label{COR6.1}
\label{9.1}
Let $\Delta = \ff^2 \Delta_K$ be an imaginary quadratic discriminant with $\Delta < -4$.  Let $M \mid N \in \Z^+$.  Let $P$ be a 
$\Delta$-CM closed point on $X_0(M,N)$.  
\begin{itemize}
\item[a)] The residue field $\Q(P)$ is isomorphic to either $\Q(M\ff)$ or $K(M \ff)$ for some $M \mid N$.
\item[b)] Let $M = \ell_1^{a_1} \cdots \ell_r^{a_r}$, $N = \ell_1^{a_1} \cdots \ell_r^{a_r}$ be the prime power decompositions of $M$ and $N$.  For $1 \leq i \leq r$, let 
$\pi_i: X_0(M,N) \ra X_0(\ell_i^{a_i'},\ell_i^{a_i})$ be the natural map and put $P_i \coloneqq \pi_i(P)$.  The following are equivalent:
\begin{itemize}
\item[(i)] The field $\Q(P)$ is formally real.
\item[(ii)] The field $\Q(P)$ does not contain $K$.
\item[(iii)] For all $1 \leq i \leq r$, the field $\Q(P_i)$ is formally real.
\item[iv)] For all $1 \leq i \leq r$, the field $\Q(P_i)$ does not contain $K$. 
\end{itemize}
\end{itemize}
\end{thm}


\subsection{Compiling Across Prime Powers with $\Delta \in \{-3,-4\}$}\label{compile_section2}

Throughout this section, we assume that $\Delta \in \{-3,-4\}$. 

\begin{prop}\label{bounds_on_fom_lemma}
Suppose that $\varphi : E \rightarrow E'$ is a cyclic $N$-isogeny, with $N$ having prime-power factorization $N = \ell_1^{a_1} \cdots \ell_r^{a_r}$. For each $i \in \{1,\ldots,r\}$, let $\varphi_i : E \rightarrow E_i$ be the $\ell_i$-primary part of $\varphi$. Let $b_i$ such that $\mathbb{Q}(\varphi_i)$ is isomorphic to either $K(\ell_i^{b_i})$ or to $\mathbb{Q}(\ell_i^{b_i})$. Then 
\[ \mathbb{Q}(\ell_1^{b_1} \cdots \ell_r^{b_r}) \subseteq \mathbb{Q}(\varphi) \subseteq K(\ell_1^{b_1} \cdots \ell_r^{b_r}). \] 
\end{prop}
\begin{proof}
Let $C = \text{ker}(\varphi)$, and for each $i \in \{1,\ldots, r\}$ let $C_i \leq C$ be the Sylow-$\ell_i$ subgroup of $C$, that is $C_i = \text{ker}(\varphi_i)$. Let $\ff$ denote the conductor of $\End(E)$, and for $1 \leq i \leq r$ let $\ff_i$ denote the conductor of 
$\End(E_i)$.  Let
\[ \mathcal{I} = \{i \mid \text{ord}_{\ell_i}(\ff_i) > \text{ord}_{\ell_i}(\ff)) \} \subseteq \{1, \ldots, r\}, \]
and let 
\[ C' = \left \langle \left \{ C_i \right \}_{i \in \mathcal{I}} \right \rangle \subseteq C. \]
Then $\varphi$ factors as $\varphi = \varphi'' \circ \varphi'$, where $\varphi ' : E \rightarrow E/C'$.  Using the fact that isogenies of 
degree prime to $\ell_i$ cannot change the $\ell_i$-part of the conductor, we see that $\End(E/C')$ has
conductor divisible by $\ell_1^{b_1} \cdots \ell_r^{b_r}$. Thus we have
\[ \mathbb{Q}(\ell_1^{b_1} \cdots \ell_r^{b_r}) \subseteq \mathbb{Q}(\varphi') \subseteq \mathbb{Q}(\varphi). \]
It remains to show the containment $\mathbb{Q}(\varphi) \subseteq K(\ell_1^{b_1} \cdots \ell_r^{b_r})$. If $j(E) \not \in \{0,1728\}$, then this follows from Theorem \ref{NEW6.1} and \cite[Prop. 3.5]{Clark22a},  so we suppose $j(E) \in \{0,1728\}$. If $j(E') = j(E)$, then 
$\varphi$ is (up to isomorphism on the target) an endomorphism of $E$, hence defined over $K$.   If $j(E') \neq j(E)$ then $j(E') \not \in \{0, 1728\}$, so our previous work applies via consideration of the dual isogeny as $\mathbb{Q}(\varphi) \cong \mathbb{Q}(\varphi^{\vee})$. 
\end{proof}

Proposition \ref{bounds_on_fom_lemma} provides bounds on the field of moduli of an isogeny. We now use this result to determine the exact field of moduli in the case where our source elliptic curve has $-3$-CM or $-4$-CM, which we state from the perspective of determining the residue field of the corresponding CM point on $X_0(N).$

\begin{thm}\label{CM_points_f=1_M=1} 
\label{PROP8.3}
Let $N \in \mathbb{Z}^+$ with prime-power factorization $\ell_1^{a_1}\cdots\ell_r^{a_r}$, and suppose $x \in X_0(N)$ is a $\Delta$-CM point with $\Delta \in \{-3,-4\}$. Let $\pi_i : X_0(N) \rightarrow X_0(\ell_i^{a_i})$ be the natural map, and let $x_i = \pi_i(x)$.  Let $P_i$ be any path in the closed point equivalence class corresponding to $x_i$ in $\mathcal{G}_{K,\ell_i,1}$, and let $d_i \geq 0$ be the number of descending edges in $P_i$. 
\begin{itemize}
\item[a)] If there is some $1 \leq i \leq r$ such that $\ell_i$ splits in $K$ and the path $P_i$ contains a surface edge, then 
\[ \mathbb{Q}(x) = K(\ell_1^{d_1} \cdots \ell_r^{d_r}). \]
\item[b)] In every other case, we have 
\[ \mathbb{Q}(x) \cong \mathbb{Q}(\ell_1^{d_1}\cdots\ell_r^{d_r}). \]
\end{itemize}
\end{thm}

\begin{proof}
Let $\varphi: E \rightarrow E'$ be a cyclic $N$-isogeny inducing the point $x$.  For each $1 \leq i \leq r$, let 
$\varphi_i: E \ra E_i$ be the $\ell_i$-primary part of $\varphi$: that is, the kernel of $\varphi_i$ is the $\ell_i$-Sylow subgroup of 
the kernel of $\varphi$. \\
Case 1: Suppose that $E'$ is also a $\Delta_K$-CM elliptic curve.  By \cite[\S 3.4]{Clark22a}, the isogeny $\varphi$ is isomorphic over $\C$ 
to $E \ra E/[I]$ for a nonzero ideal $I$ of $\Z_K$, and we have $\Q(\varphi) \cong \Q(j(E)) = \Q$ if $I$ is real ideal (i.e., $I = \overline{I})$ and $\Q(\varphi) = K(j(E)) = K$ is $I$ is not a real ideal.  If we factor $I = \pp_1^{c_1} \cdots \pp_r^{c_r}$ into prime powers and 
$\pp_i$ lies over $\ell_i$, then we have (up to an isomorphism on the target) that $\varphi_i: E \ra E/[\pp_i^{c_i}]$.  
 Notice that the path in $\mathcal{G}_{K,\ell_i,1}$ corresponding to $\varphi_i$ lies 
entirely on the surface.   If some $\ell_i$ splits in $K$, then $\pp_i^{c_i}$ is not a real ideal, so $\Q(\varphi) = K$.  If no $\ell_i$ splits in $K$, then  each $\pp_i^{c_i}$ is real, so $I$ is real and $\Q(\varphi) = \Q$.  \\
Case 2: Otherwise $E'$ is a $\Delta = \ff^2 \Delta_K$-CM elliptic curve for some $\ff > 1$.  Since $\Q(\varphi) = \Q(\varphi^{\vee})$, 
we may compute the field of moduli of the dual isogeny $\varphi^{\vee}: E' \ra E$.  Since $\Aut E' = \{ \pm 1\}$, the rationality of 
a subgroup of $E'$ is independent of the model, so we have $\Q(\varphi) = \Q(\varphi_1) \cdots \Q(\varphi_r)$, and the result 
now follows from \cite[Thm. 5.1]{Clark22a}.
\end{proof}

\begin{cor}
\label{NEWLOWERBOUNDCOR}
Let $\varphi: E \ra E'$ be an isogeny of $K$-CM elliptic curves defined over $\C$.  Let $\ff$ (resp. $\ff'$) be the conductor of $\End(E)$ 
(resp. of $\End(E')$).  Then the field of moduli $\Q(\varphi)$ of $\varphi$ contains a subfield isomorphic to $\Q(\lcm(\ff,\ff'))$.
\end{cor}
\begin{proof}
Theorems \ref{COR6.1} and \ref{PROP8.3} imply that $\Q(\varphi)$ is isomorphic to $\Q(M\ff)$ or $K(M\ff)$ for some $M \in \Z^+$. 
Using the fact that $\Q(\varphi) = \Q(\varphi^{\vee})$ we find that $\ff' \mid M \ff$, and the result follows.
\end{proof}
\noindent
When $\Delta_K < -4$, then Corollary \ref{NEWLOWERBOUNDCOR} holds just because $\Q(\iota) \supseteq \Q(j(E),j(E'))$.  
However:

\begin{cor}
\label{KADJOINJCOR}
Let $\Delta_K \in \{-3,-4\}$, let $\ff,\ff'$ be coprime positive integers not lying in the set $S$ of Proposition \ref{SAIAPROP1}: that is, if $\Delta_K = -3$ then $\ff,\ff' > 3$ and if $\Delta_K = -4$ then $\ff,\ff' > 2$.  Let $\varphi: E \ra E'$ be an isogeny of 
$K$-CM elliptic curves defined over $\C$ such that $\End(E)$ has conductor $\ff$ and $\End(E')$ has conductor $\ff'$.  Then 
$\varphi$ cannot be defined over $K(j(E),j(E'))$.
\end{cor}
\begin{proof}
This follows from Corollary \ref{NEWLOWERBOUNDCOR} and Proposition \ref{SAIAPROP1}.
\end{proof} 
\noindent
Next we obtain a version of Theorem \ref{CM_points_f=1_M=1} in the $M \geq 2$ case, finding in particular that the residue 
field of a CM point on $X_0(M,N)$ is isomorphic to either a rational ring class field or a ring class field in all cases.

\begin{thm}
\label{THM8.4}
Let $M,N \in \Z^{\geq 2}$ with $M \mid N$, and write 
\[ M = \ell_1^{a_1'} \cdots \ell_r ^{a_r'}, \ N = \ell_1^{a_r} \cdots \ell_r^{a_r}. \]
Let 
\[ \alpha: X_0(M,N) \ra X_0(M,M), \ \beta: X_0(M,N) \ra X_0(N) \]
and 
\[ \forall 1 \leq i \leq r, \ \pi_i: X_0(N) \ra X_0(\ell_i^{a_i}) \]
be the canonical maps.  Let $\Delta \in \{-3,-4\}$, let $P$ be a $\Delta$-CM closed point of $X_0(M,N)$ and let $y_i \coloneqq \pi_i(\beta(P))$.  Let $d_i$ 
be the number of descending edges in $y_i$, and put
\[ \overline{d_i} \coloneqq \max(a_i',d_i). \]
\begin{itemize}
\item[a)] Suppose $M \geq 3$ or ($M = 2$ and $\Delta = -3$).  Then 
\[ \Q(P) = K(\ell_1^{\overline{d_1}} \cdots \ell_r^{\overline{d_r}}). \]
\item[b)] Suppose $M = 2$ and $\Delta = -4$; we put $\ell_1 = 2$.   
\begin{itemize}
\item[(i)] If $a_1 \geq 2$ and $y_1 \in X_0(2^a)$ is not purely descending, then 
\[ \Q(P) =K(\ell_1^{\overline{d_1}} \cdots \ell_r^{\overline{d_r}}). \]
\item[(ii)] Otherwise, we have
\[ \Q(P) \cong \begin{cases} K(2^{a_1} \cdot \ell_2^{d_2} \cdots \ell_r^{d_r}) \quad &\text{ if } K \subseteq \mathbb{Q}(y_i) \text{ for some } 2 \leq i \leq r, \\
\mathbb{Q}(2^{a_1} \cdot \ell_2^{d_2} \cdots \ell_r^{d_r}) \quad &\text{ otherwise.} \end{cases} \]
\end{itemize}
\end{itemize}
\end{thm}
\begin{proof}
a) Our hypotheses on $M$ imply (cf. \S 6) that
\[ \Q(\alpha(P)) = K(\ell_1^{a_1'} \cdots \ell_r^{a_r'}). \]
In particular, this implies that $\Q(P) \supset K$.  Applying Proposition \ref{PROP8.3} we get
\[ K(\beta(P)) = K(\ell_1^{d_1} \cdots \ell_r^{d_r}). \]
Using this and Proposition \ref{SAIAPROP1}a), we get:
\[ \Q(P) \supseteq \Q( \alpha(P), \beta(P)) = K(\alpha(P),\beta(P)) = K(\ell_1^{a_1'} \cdots \ell_r^{a_r'})K(\ell_1^{d_1} \cdots \ell_r^{d_r}) = 
K(\ell_1^{\overline{d_1}} \cdots \ell_r^{\overline{d_r}}). \] 
Conversely, using Proposition \ref{REFINEDBC4.5}, let $E_{/K(\ell_1^{\overline{d_1}} \cdots \ell_r^{\overline{d_r}})}$ be an elliptic curve with $j$-invariant $j_{\Delta}$ 
and with modulo $\ell_1^{\overline{d_1} \cdots \overline{d_r}}$-Galois representation given by scalar matrices.
For $1 \leq i \leq r$, let $\varphi_i: E \ra E_i$ be a cyclic 
$\ell_i^{a_i}$-isogeny of $\C$-elliptic curves containing $d_i$ descending edges.  It follows from the proof of Theorem \ref{FOMTHM}, 
that the kernel $C_i$ of $\varphi_i$ is a $K(\ell_1^{\overline{d_1}} \cdots \ell_r^{\overline{d_r}})$-rational subgroup scheme, hence 
so is $C = \langle C_1,\ldots,C_r \rangle$, which is the kernel of $\varphi$.  Thus we have found a $K(\ell_1^{\overline{d_1}} \cdots 
\ell_r^{\overline{d_r}})$-rational model of $P$. \\
b) (i) By Corollary \ref{FOMCOR}, we have $\Q(P) \supseteq K$.   The rest of the argument is the same as that 
of part a).  \\
(ii) First, suppose that $y_1$ is purely descending and let $\varphi: E \ra E'$ be an isogeny defined over $\Q(\beta(P))$ that induces the point $\pi_1(\beta(P)) \in X_0(2^{a_1})$.  Then the initial 
edge of $\varphi$ is downward, so as in the proof of Corollary \ref{FOMCOR} every order $2$ subgroup of $E$ is $\Q(\beta(P))$-rational. This provides $\mathbb{Q}(P) = \mathbb{Q}(\beta(P))$, and the stated isomorphism then follows from Proposition \ref{PROP8.3}. 

Lastly, suppose that $a_1 = 1$ and that the $2$-isogeny inducing $y_1$ is horizontal.  This final case can be reduced to the previous one via automorphisms of the modular curve $X_0(2,N)$.  Indeed, observe that the map 
$\pi: X_0(2,N) \ra X_0(\frac{N}{2})$ is a $\GL_2(\Z/2\Z) \cong S_3$ Galois cover.  The level $N$-structure associated to $X_0(2,N)$ may viewed as an elliptic curve equipped with an ordered triple $(C,C_1,C_2)$, where $C$ is a cyclic subgroup of $E$ of order $\frac{N}{2}$ 
and $C_1$ and $C_2$ are distinct cyclic subgroups of $E$ of order $2$.
The map $\pi$ can then be viewed as $(E,C,C_1,C_2) \mapsto (E,\langle C,C_1 \rangle)$.  The $GL_2(\Z/2\Z)$ action fixes $E$ and $C$,
and on the pair $(C_1,C_2)$ the action is the natural and simply transitive one on pairs of order $2$ subgroups of $E[2]$.  Therefore if $P \in X_0(2,N)$ is induced by a tuple $(E,C,C_1,C_2)$ with $C_1$ a horizontal $2$-isogeny, then $C_2$ is necessarily descending and we have $\Q(P) = \Q(P')$ 
where $P'$ is induced by $(E,C,C_2,C_1)$.  By the previous case, we have that
\[ \Q(P') = \Q(\beta(P')) \]
is the field of moduli of the isogeny $E \ra E/\langle C,C_2 \rangle$. Again, the stated isomorphism follows via Proposition \ref{PROP8.3}.  
\end{proof}

\noindent
Theorems \ref{PROP8.3} and \ref{THM8.4} are the key ingredients for the determination of primitive residue fields and primitive degrees of $\Delta$-CM 
points on $X_0(M,N)$, which we will provide in the next section.  However, there remains the problem of computing the set of all $\Delta$-CM 
points on $X_0(M,N)$ with a given rational ring class field or ring class field as residue field.  The following results solve this problem. 

\begin{thm}
Let $N \in \Z^{\geq 2}$ have prime power factorization $N = \ell_1^{a_1} \cdots \ell_r^{a_r}$.  For $1 \leq i \leq r$, let $P_i \in X_0(\ell_i^{a_i})$ be a $\Delta$-CM point, and let $\pi_i: X_0(N) \ra X_0(\ell_i^{a_i})$ denote the natural map.  Let $\mathcal{F}$ be the set of closed points $P \in X_0(N)$ such that $\pi_i(P) = P_i$ for all $1 \leq i \leq r$.  Put 
\[ s \coloneqq \# \{1 \leq i \leq r \mid \Q(P_i) \text{ contains } K \}. \]
If $s = 0$, then $\# \mathcal{F} = 1$.  If $s \geq 1$, then $\mathcal{F}$ consists of $2^{s-1}$ points, each with residue field the same 
ring class field.
\end{thm}
\begin{proof}
Step 1: Suppose that $\Delta = \ff^2 \Delta_K < -4$.  Let $F$ be the fiber of $X_0(N) \ra X(1)$ over $J_{\Delta}$, and for $1 \leq i \leq r$ let $F_i$ be the fiber of $X_0(\ell_i^{a_i}) \ra X(1)$ over $J_{\Delta}$.  Then by \cite[Prop. 3.5]{Clark22a} we have that $F$ is the fiber 
product of $F_1,\ldots,F_r$ over $\Spec \Q(\ff)$.  By our hypothesis, we have either $\Delta_K < -4$ or $\ff > 1$.   If $\Delta_K < -4$,
the result follows from this and \cite[Prop. 2.10]{Clark22a}, as is recorded in \cite[\S 9.1]{Clark22a}.  If $\Delta_K \in \{-3,-4\}$ 
and $\ff > 1$, the result follows from this and Proposition \ref{SAIAPROP2}.  \\ \indent
For the remainder of the argument we suppose that $\Delta \in \{-3,-4\}$.   Let $P \in \mathcal{F}$, let $\varphi: E \ra E'$ be 
an isogeny inducing the point $P$, and put $C \coloneqq \Ker \varphi^{\vee}$.  The endomorphism ring of $E'$ has 
discriminant $(\ff')^2 \Delta_K$ for some $\ff' \mid N$.  \\ \
Step 2: We suppose that $\ff' > 1$.  For each $1 \leq i \leq r$, let $C_i \coloneqq C[\ell_i^{a_i}]$ and $(\varphi^{\vee})_i$ be the 
isogeny $E' \ra E'/C_i \eqqcolon E_i$.  Then $\varphi^{\vee}$ factors as $\psi_i \circ (\varphi^{\vee})_i$, wehre $\psi_i: E_i \ra E$ is a cyclic $\frac{N}{\ell_i^{a_i}}$-isogeny.  Let $\ff_i$ be the conductor of the endomorphism ring of $E_i$, so $\ord_{\ell_i}(\ff_i) = 0$, since 
the conductor of $\End(E)$ is $1$ and $\deg(\psi_i)$ is prime to $\ell_i$.  
\\ \indent
For $1 \leq i \leq r$, the path in $\mathcal{G}_{K,\ell_i,1}$ corresponding to $(\varphi^{\vee})_i$ therefore terminates at the unique 
surface vertex, hence it consists of $\ord_{\ell_i}(\ff')$ ascending edges, which are uniquely determined by $E'$, followed by 
$a_i - \ord_{\ell_i}(\ff_i)$ horizontal edges.  If $\ell_i$ does not split in $K$ the path corresponding to $(\varphi^{\vee})_i$ is 
therefore uniquely determined, whereas if $\ell_i$ splits in $K$ there are two such paths which are complex conjugates of each other and therefore determine the same closed point equivalence class.  Let $P_i' \in X_0(N)$ be the $\ff^2 \Delta_K$-CM point induced by 
$(\varphi^{\vee})_i$, and let $\mathcal{F}'$ be the set of closed points $P' \in X_0(N)$ such that $\pi_i(P') = P_i'$ for all $1 \leq i \leq r$.
Thus passage to the dual isogeny gives a residue-field preserving bijection from $\mathcal{F}$ to $\mathcal{F}'$.  Let $1 \leq i \leq r$.  
Because the path corresponding to $P_i$ begins at the surface and the path corrsesponding to $P_i'$ ends at the surface, we have that $\Q(P_i)$ contains $K$ if and only if $\ell_i$ splits in $K$ in $\ord_{\ell_i}(\ff') < a_i$ if and only if $\Q(P_i')$ contains $K$.  
Applying Step 1, we get that if $s = 0$ then $\# \mathcal{F} = \# \mathcal{F}' = 1$, while if $s \geq 1$ then 
\[ \# \mathcal{F} = \# \mathcal{F}' = 2^{s-1}. \]
Step 3: We suppose that $\ff' = 1$.  In this case every element of $\mathcal{F}$ is induced by an isogeny $\varphi_I: E \ra E/E[I]$ for 
$I$ a nonzero $\Z_K$-ideal such that $\Z_K/I$ is cyclic, and $\deg \varphi_I = ||I|| \coloneqq \# \Z_K/I$.  Moreover the field of moduli of $\varphi_I$ is $\Q$ if $I = \overline{I}$ 
and $K$ otherwise \cite[\S 3.4]{Clark22a}.  For distinct $I$ and $J$, the isogenies $\varphi_I$ and $\varphi_J$ induce the same closed point 
on $X_0(N)$ if and only if $J = \overline{I}$, so closed point equivalence classes in this case correspond to orbits under $\gg_{\R}$.  
For a prime power $\ell^a$, there is an ideal $I$ of norm $\ell^a$ such that $\Z_K/I$ is cyclic if and only if ($\ell$ ramifies in $\Z_K$ 
and $a  = 1$) or $\ell$ splits in $K$.  In the ramified case there is a unique ideal of norm $\ell$, while in the split case the two ideals 
of norm $\ell^a$ are $\pp^a$ and $\overline{\pp}^a$ where $\pp$ and $\overline{\pp}$ are the two primes of $\Z_K$ lying over $\ell$.  
If $s$ is the number of $1 \leq i \leq r$ such that $\ell_i$ splits in $\Z_K$, then if $s = 0$ then $N$ the unique prime number that ramifies in $\Z_K$ so $\# \mathcal{F} = 1$.  If $s \geq 1$, then the number of ideals $I$ of norm $N$ such that $\Z_K/I$ is cyclic is $2^s$, and the 
number of $\gg_{\R}$-orbits of such ideals is $2^{s-1}$.  
\end{proof}

\begin{cor}
Let $M \mid N$ be in $\mathbb{Z}^{+}$ with prime-power factorizations $M = \ell_1^{a_1'} \cdots \ell_r^{a_r'}$ and $N = \ell_1^{a_1} \cdots \ell_r^{a_r}$ with $\ell_1 \leq \ldots \leq \ell_r$. Let $\pi : X_0(M,N) \rightarrow X_0(N)$ denote the natural map, and let $\pi_i : X_0(N) \rightarrow X_0(\ell_1^{a_1})$ denote the natural map for $1 \leq i \leq r$. Let $P_i \in X_0(\ell_1^{a_1})$ be $\Delta$-CM points for each index $i$, and let $\mathcal{F}$ be the set of closed points $P \in X_0(M,N)$ such that $\pi_i(\pi(P)) = P_i$ for all $1 \leq i \leq r$. For each $1 \leq i \leq r$, let $d_i$ denote the number of descending edges occurring in a path in $\mathcal{G}_{K,\ell_i,f_0}$ corresponding to $P_i$. Put
\[ s := \#\{1 \leq i \leq r \mid \mathbb{Q}(P_i) \text{ contains } K\}, \]
and put 
\[ \epsilon = \begin{cases} 1 \quad &\text{ if } s=0, (M, \Delta) = (2, -4) \text{ and } a_1 \not\in \{1,d_1\} ,  \\
0 \quad &\text{ otherwise}. \end{cases} \]
We then have that $\mathcal{F}$ consists of 
\begin{align*} 
\# \mathcal{F} = 2^{\max(s-1,0) - \epsilon} \cdot M \varphi(M) \cdot \left( \prod_{i \text{ with } a_i' > d_i = 0} \ell_i^{a_i'-1} \left(\ell_i - \leg{\Delta}{\ell_i}\right)\right)^{-1} \cdot \left( \prod_{i \text{ with } a_i' > d_i > 0} \ell_i^{a_i'-d_i} \right)^{-1} 
\end{align*}
points with isomorphic residue fields. 
\end{cor}
\begin{proof}
Given a point $P' \in X_0(N)$, Theorem \ref{THM8.4} determines the residue field of any point $Q \in \pi^{-1}(P')$ in terms of $P'$ and $M$, so we know that each point in $\mathcal{F}$ has the same residue field up to isomorphism.   If $\Delta < -4$ then the map $\pi$ is unramified.  If $\Delta = -4$ (resp. $\Delta = -3$), then the map $\pi$ has ramification index $1$ if and only if the path in $\mathcal{G}_{K,\ell,1}$ corresponding to $x_i$ is purely horizontal for each $i$. Therefore, because $M \geq 2$ (see the discussion in \S \ref{X_1_sec} for more details), the point $P$ necessarily has ramification index $e = 2$ (resp. $e=3$) exactly with respect to $\pi$ in this situation, and otherwise has ramification index $1$. In any event, letting $P' \in X_0(N)$ be a point with $\pi_i(P') = P_i$ for all $1 \leq i \leq r$, we must have that the number of points $P \in \mathcal{F}$ lying above $P'$ is 
\[ \dfrac{\text{deg}(\pi)}{e \cdot [\mathbb{Q}(P) : \mathbb{Q}(\pi(P))]} = \dfrac{M  \varphi(M)}{e \cdot [\mathbb{Q}(P) : \mathbb{Q}(\pi(P))]}\]
The $\epsilon = 1$ case, by Theorem \ref{THM8.4}, is exactly the case in which $\mathbb{Q}(P')$ is isomorphic to a rational ring class field while $\mathbb{Q}(P)$ is a ring class field. In the $\epsilon = 0$ case, we then have 
\begin{align*}  e\cdot [\mathbb{Q}(P) : \mathbb{Q}(\pi(P))] &= e \cdot \left[\mathbb{Q}\left(\ell_1^{\text{max}\{a_1',d_1\}} \cdots \ell_1^{\text{max}\{a_1',d_1\}}\right) : \mathbb{Q}\left(\ell_1^{d_1} \cdots \ell_r^{d_r}\right) \right) \\
&= \left( \prod_{i \text{ with } a_i' > d_i = 0} \ell_i^{a_i'-1} \left(\ell_i - \leg{\Delta_K}{  \ell_i}\right)\right) \cdot \left( \prod_{i \text{ with } a_i' > d_i > 0} \ell_i^{a_i'-d_i} \right),
\end{align*}
while the only change in this quantity in the $\epsilon = 1$ case is an additional factor of $2$.  This combined with the result of the previous theorem gives the result as stated.
\end{proof}

\subsection{Primitive Residue Fields and Primitive Degrees I}
In this section and the next, we extend the results of \cite[\S 9.2-9.3]{Clark22a} to handle $\Delta_K \in \{-3, -4\}$. Given our extensions of the results on primitive residue fields of $\Delta$-CM points on $X_0(\ell^{a_1'},\ell^{a_1})$ for $\ell$ prime and on compiling across prime powers this proceeds nearly exactly as therein. 

In this section, as in \cite[\S 9.2]{Clark22a}, we suppose that either $M=1$ or that ($M=2$ and $\Delta$ is even). This assumption implies that there is a closed $\Delta$-CM point on $X_0(M,N)$ with residue field isomorphic to $\Q(N \ff)$, and therefore there is a unique $B \mid N$ such that $\mathbb{Q}(B \ff)$ is a primitive residue field of $\Delta$-CM points on $X_0(M,N)$. For each $1 \leq i \leq r$, take $b_i$ to be the least integer $B_i$ such that $\Q(\ell_i^{B_i} \ff)$ 
is isomorphic to the residue field of 
a $\Delta$-CM point on $X_0(\ell_i^{a_i'},\ell_i^{a_i})$. We then have
\[ B = \ell_1^{b_1} \cdots \ell_r^{b_r}. \]
 There is at most 
one other primitive residue field of a $\Delta$-CM point on $X_0(M,N)$, and there is one other exactly when there are two primitive residue fields for $\Delta$-CM points on $X_0(\ell_i^{a_i'},\ell_i^{a_i})$ for some 
$1 \leq i \leq r$. In this case, letting $c_i$, for $1 \leq i \leq r$, be the least natural number $C_i$ such that there is
a $\Delta$-CM point on $X_0(\ell_i^{a_i'},\ell_i^{a_i})$ with residue field isomorphic to either $\Q(\ell_i^{C_i} \ff)$ or to $K(\ell_i^{C_i} \ff)$, we have that the other primitive residue field is $K(C \ff)$, where
 \[ C = \ell_1^{c_1} \cdots \ell_r^{c_r}. \]
 If there is a unique primitive residue field of $\Delta$-CM points on $X_0(M,N)$, then of course there is a unique primitive degree $[\mathbb{Q}(P) : \mathbb{Q}]$ of such points. Supposing we are in the case of two primitive residue fields $\Q(B\ff)$ and $K(C\ff)$, we 
put 
\[  \bb \coloneqq [\Q(B\ff):\Q] \quad \text{ and } \quad \cc \coloneqq [K(C\ff):\Q] .\]
We will have a unique primitive degree if and only if one of $\bb$ and $\cc$ divides the other, and we will soon see that we always have $\cc \leq \bb$, so the question is whether $\cc \mid \bb$. This divisibility certainly holds if the analogous divisibility holds at every prime power, but as seen in \cite{Clark22a} this is not a necessary condition. The following theorem determines exactly the situation, generalizing \cite[Thm. 9.2]{Clark22a}. The proof is only mildly more complicated than the proof of this prior result, owing to handling the $\Delta = -4$ case. 

\newcommand{\ibullet}{i_{\bullet}}
\newcommand{\jbullet}{j_{\bullet}}
\begin{thm}
\label{COOLTECHTHM}
Let $\Delta = \ff^2 \Delta_K$ be an imaginary quadratic discriminant, and let $M = \ell_1^{a_1'} \cdots \ell_r^{a_r'} \mid N = \ell_1^{a_1} \cdots \ell_r^{a_r}$.  We suppose that either $M = 1$ or ($M=2$ and $\Delta$ is even).  For $1 \leq i \leq r$, let 
$b_i \geq 0$ be the unique natural number such that $\Q(\ell_i^{b_i} \ff)$ occurs up to isomorphism as a primitive residue field of a closed $\Delta$-CM point on $X_0(\ell_i^{a_i'},\ell_i^{a_i})$.  Let $c_i$ be 
equal to $b_i$ if there is a unique primitive residue field of $\Delta$-CM points on $X_0(\ell_i^{a_i'},\ell_i^{a_i})$ and otherwise let it be such that the unique non-real primitive residue field of a closed $\Delta$-CM point 
on $X_0(\ell_i^{a_i'},\ell_i^{a_i})$ is $K(\ell_i^{c_i} \ff)$.  Put $B \coloneqq \ell_1^{b_1} \cdots \ell_r^{b_r}$ and $C \coloneqq \ell_1^{c_1} \cdots \ell_r^{c_r}$.  Let $s$ be the number of $1 \leq i \leq r$ such that 
there is a non-real primitive residue field of a closed $\Delta$-CM point on $X_0(\ell_i^{a_i'},\ell_i^{a_i})$.   
\begin{itemize}
\item[a)] If $s = 0$, the unique primitive residue field of a $\Delta$-CM point on $X_0(M,N)$ is $\Q(B\ff)$, so the unique primitive degree of a $\Delta$-CM point on $X_0(M,N)$ is $[\Q(B \ff):\Q]$.
\item[b)] If $s \geq 1$ and there is some $1 \leq i \leq r$ such that there are two primitive residue fields of closed $\Delta$-CM points on $X_0(\ell_i^{a_i'},\ell_i^{a_i})$ and we are not in Case 1.5b) with respect to $\Delta$ and $\ell_i^{a_i}$,
then:
\begin{itemize}
\item[(i)] There are two primitive residue fields of $\Delta$-CM points on $X_0(M,N)$: $\Q(B \ff)$ and $K(C \ff)$. 
\item[(ii)] The unique primitive degree of $\Delta$-CM points on $X_0(M,N)$ is $[K(C\ff):\Q]$.  
\end{itemize}
\item[c)] If $s \geq 1$ and for all $1 \leq i \leq r$ such that there are two primitive residue fields of closed $\Delta$-CM points on $X_0(\ell_i^{a_i'},\ell_i^{a_i})$ we are in Case 1.5b), 
then there are two primitive degrees of $\Delta$-CM points on $X_0(M,N)$: $[\Q(B \ff):\Q]$ and $[K(C \ff):\Q]$.  
\end{itemize}
\end{thm}
\begin{proof}
The case $s = 0$ is immediate from the above discussion.  Henceforth we suppose $s \geq 1$. We then have (up to isomorphism) two primitive 
residue fields of $\Delta$-CM closed points on $X_0(M,N)$: $\Q(B \ff)$ and $K(C \ff)$, and as above we put
\[ \bb \coloneqq [\Q(B \ff) : \mathbb{Q}], \ \cc \coloneqq [K(C \ff) : \mathbb{Q}]. \]
For each $1 \leq i \leq r$, let $F_i$ be a primitive residue field of a closed point of a $\Delta$-CM elliptic curve on $X_0(\ell_i^{a_i'},\ell_i^{a_i})$; if there is any non-real such field, take $F_i$ to be nonreal.  Note that for each $i$ such that there are two primitive residue fields $\Q(\ell_i^{b_i} \ff)$ and $K(\ell_i^{c_i} \ff)$ we have $[K(\ell_i^{c_i} \ff):\Q] \leq [\Q(\ell_i^{b_i} \ff):\Q$]. By Propositions \ref{SAIAPROP1} and \ref{SAIAPROP2}, there is $0 \leq r' \leq r-1$ such that 
\begin{align*}
2^{s-1} \cdot [K(C\ff) : \mathbb{Q}(\ff)] &= \left( \frac{\omega_K}{2}\right)^{r'} \cdot \left( \dim_{\Q(\ff)} F_1 \otimes_{\Q(\ff)} \otimes \cdots \otimes_{\Q(\ff)} F_r \right)\\
&\leq \left( \frac{\omega_K}{2}\right)^{r'} \cdot \left( \dim_{\Q(\ff)} \Q(\ell_1^{b_1} \ff) \otimes_{\Q(\ff)} \cdots \otimes_{\Q(\ff)} \Q(\ell_r^{b_r} \ff) \right)\\
&= [\Q(\ell_1^{b_1} \ff) \cdots \Q(\ell_r^{b_r}) : \mathbb{Q}(\ff)]. 
\end{align*}
It follows that $\cc \leq \bb$.  Thus there is a unique primitive degree exactly when $\cc \mid \bb$, as claimed in the above discussion. 

Because $K(C \ff) \subseteq K(B \ff) = K \Q(B \ff)$, we have $\mathbf{c} \mid 2 \mathbf{b}$.  In particular, we have $\ord_p(\mathbf{c}) \leq \ord_p(\mathbf{b})$ 
for every odd prime $p$.  \\
\textbf{Case 1}: Suppose $\Delta_K \neq -4$.  By Proposition \ref{TEDIOUSALGEBRAPROP1} we have 
\[ \ord_2(\mathbf{c})= 1 + \ord_2([\Q(C\ff):\Q(\ff)]) = 1 + \sum_{i=1}^r\  [\Q(\ell^{c_i} \ff):\Q(\ff)] \]
\[ \ord_2(\mathbf{b}) = \ord_2([\Q(B \ff):\Q(\ff)]) = \sum_{i=1}^r \ [\Q(\ell^{b_i} \ff:\Q(\ff)]. \]
It follows that $\mathbf{c} \mid \mathbf{b}$ if and only if there is some $1 \leq i \leq r$ such that there are two primitive residue fields of $\Delta$-CM closed points on $X_0(\ell_i^{a_i'},\ell_i^{a_i})$ 
for which we have 
\[\ord_2([\Q(\ell^{c_i} \ff):\Q(\ff)]) < \ord_2([\Q(\ell^{b_i} \ff):\Q(\ff)], \]
which holds if and only if
\[\ord_2([K(\ell^{c_i} \ff):\Q(\ff)]) \leq \ord_2([\Q(\ell^{b_i} \ff):\Q(\ff)]. \]
This holds in every case in which there are two primitive residue fields \emph{except} Case 1.5b).

\noindent
\textbf{Case 2:} Suppose $\Delta_K = -4$. Let $r_\mathbf{c}$ be the number of indices $1 \leq i \leq r$ such that ${\ell_i}^{2c_i}\Delta_K \not \in \{-4,-16\}$, and let $r_\mathbf{b}$ be the number of indices $1 \leq i \leq r$ such that ${\ell_i}^{2b_i}\Delta_K \not \in \{-4,-16\}$. We have $0 \leq r_\mathbf{c} \leq r_\mathbf{b} \leq r$. Proposition \ref{SAIAPROP1} then gives:
\begin{align*}
\text{ord}_2(\mathbf{c}) &= 1 + \text{ord}_2\left( [\mathbb{Q}(C) : \mathbb{Q}] \right) \\
&= 1 + \text{ord}_2\left( [\mathbb{Q}(C) : \mathbb{Q}({\ell_1}^{c_1})\cdots \mathbb{Q}({\ell_r}^{c_r} )] \right) + \text{ord}_2\left( [\mathbb{Q}({\ell_1}^{c_1})\cdots \mathbb{Q}({\ell_r}^{c_r} ) : \mathbb{Q}] \right) \\
&= r_{\cc} + \sum_{i=1}^{r} \text{ord}_2\left([\mathbb{Q}({\ell_i}^{c_i}) : \mathbb{Q}]\right)
\end{align*}
and 
\begin{align*}
\text{ord}_2(\bb) &= \text{ord}_2\left( [\mathbb{Q}(B) : \mathbb{Q}] \right) \\
&= \text{ord}_2\left( [\mathbb{Q}(B) : \mathbb{Q}({\ell_1}^{b_1})\cdots \mathbb{Q}({\ell_r}^{b_r} )] \right) + \text{ord}_2\left( [\mathbb{Q}({\ell_1}^{b_1})\cdots \mathbb{Q}({\ell_r}^{b_r} ) : \mathbb{Q}] \right) \\
&= r_{\bb} - 1 + \sum_{i=1}^{r} \text{ord}_2\left([\mathbb{Q}({\ell_i}^{b_i}) : \mathbb{Q}]\right).
\end{align*}
We see then that $\mathbf{c} \mid \mathbf{b}$ if and only if 
\[ \sum_{i=1}^r \text{ord}_2\left( [\mathbb{Q}({\ell_i}^{c_i}) : \mathbb{Q}]\right) < r_\mathbf{b} - r_\mathbf{c} + \sum_{i=1}^r \text{ord}_2\left( [\mathbb{Q}({\ell_i}^{b_i}) : \mathbb{Q}]\right). \]
We find that $\cc \mid \bb$ if and only if $r_{\bb} > r_{\cc}$ or there is some $1 \leq i \leq r$ such that $c_i < b_i$ for which we are 
\emph{not} in Case 1.5b) with respect to $\Delta$ and $\ell_i^{a_i}$.  Comparing with the statement of the result, we must show: 
in every case in which $r_{\bb} > r_{\cc}$ there is some $1 \leq i \leq r$ for which $c_i < b_i$ and we are not in Case 1.5b) for 
$\Delta$ and $\ell_i^{a_i}$. So: \\
$\bullet$ If $\Delta \notin \{-4,-16\}$, then $r_{\bb} = r_{\cc}$, so there are two primitive degrees if and only if we are in Case 1.5b) 
for all $1 \leq i \leq r$ for which $c_i < b_i$, as claimed.  \\
$\bullet$ If $\Delta \in \{-4,-16\}$ then Case 1.5b) cannot occur for any $i$ and thus $\cc \mid \bb$, as claimed.
\end{proof}

\subsection{Primitive Residue Fields and Primitive Degrees II} In this section we treat the case in which either $M \geq 3$, or ($M = 2$ and $\Delta$ is odd).  Thanks to the work of \S 7, this case follows exactly as in \cite[\S 9.3]{Clark22a}. In particular, our assumptions imply that there is a unique primitive residue 
field, which is a ring class field $K(C\ff)$.  

Let $M = \ell_1^{a_1'} \cdots \ \ell_r^{a_r'}$ and $N = \ell_1^{a_1} \cdots \ell_r^{a_r}$.  For an index $i \in \{1,\ldots, r\}$, if the only primitive residue field of a $\Delta$-CM point on 
$X_0(\ell_i^{a_i'},\ell_i^{a_i})$ is $\Q(\ell^c \ff)$ then put $c_i \coloneqq c$.  Otherwise, the primitive residue fields of $\Delta$-CM points on $X_0(\ell_i^{a_i'},\ell_i^{a_i})$ are of the form $\Q(\ell^b \ff)$ and $K(\ell^c \ff)$, and we put $c_i \coloneqq c$. We then have
\[ C = \ell_1^{c_i} \cdots \ell_r^{c_r}. \]

\end{document}